\documentclass[a4paper]{article}
\pdfoutput=1
\usepackage[utf8]{inputenc}
\usepackage[T1]{fontenc}
\usepackage{textcomp} 

\usepackage[english]{babel}
\usepackage{mathtools,amssymb,amsthm}
\usepackage{enumitem}

\usepackage{array}

\usepackage{graphicx}
\usepackage{xcolor}
\usepackage{float} 
\usepackage{extarrows}
\usepackage{tikz}
\usetikzlibrary{arrows.meta}
\usepackage[framemethod=tikz]{mdframed}

\usepackage{cite}
\usepackage[scaled]{helvet}
\renewcommand{\emph}[1]{\textbf{#1}}

\usepackage{thmtools}
\declaretheoremstyle[
	spaceabove=\topsep, spacebelow=\topsep,
	headfont=\normalfont\bfseries,
	postheadspace=.5em,
	qed=$\square$
	]{whiteqed}

\declaretheorem[style=definition,parent=section,title=Theorem]{Thm}
\declaretheorem[style=definition,sibling=Thm,title=Proposition]{Prop}
\declaretheorem[style=definition,sibling=Thm,title=Lemma]{Lem}
\declaretheorem[style=definition,sibling=Thm,title=Corollary]{Cor}
\declaretheorem[style=definition,numbered=no]{Claim}

\declaretheorem[style=whiteqed,sibling=Thm,title=Definition]{Def}

\declaretheorem[style=whiteqed,sibling=Thm,title=Remark]{Rmk}
\declaretheorem[style=whiteqed,sibling=Thm,title=Example]{Expl}

\declaretheorem[style=whiteqed,sibling=Thm,title=Theorem]{Thm*}
\declaretheorem[style=whiteqed,sibling=Thm,title=Proposition]{Prop*}
\declaretheorem[style=whiteqed,sibling=Thm,title=Corollary]{Cor*}
\declaretheorem[style=whiteqed,sibling=Thm,title=Lemma]{Lem*}

\newcommand{\defarrow}{\; \xLeftrightarrow{\mathrm{def.}} \;}

\usepackage{bm}
\newcommand{\bvec}[1]{\bm{#1}}

\newcommand{\smid}{\mathrel{;}}
\mathchardef\mhyphen="2D
\usepackage{stmaryrd}
\newcommand{\dbracket}[1]{\left \llbracket #1 \right \rrbracket}
\newcommand{\dvert}[1]{\left \lvert #1 \right \rvert}
\newcommand{\dVert}[1]{\left \lVert #1 \right \rVert}
\newcommand{\Set}[1]{\left \{\, #1 \, \right \}}
\let\imply=\rightarrow
\let\limply=\rightarrow

\newcommand{\id}{\mathrm{id}}
\newcommand{\op}{\mathrm{op}}
\newcommand{\isoarrow}{\stackrel{\sim}{\to}}

\DeclareMathOperator{\Hom}{Hom}

\DeclareMathOperator*{\rlim}{\underrightarrow{\lim}}
\DeclareMathOperator{\Sub}{Sub}

\allowdisplaybreaks[3]

\usepackage[bookmarks=true,bookmarksnumbered=true]{hyperref}

\usepackage[affil-it]{authblk}

\title{\Large Sheaves of Structures, Heyting-Valued Structures, and a Generalization of {\L}o\'{s}'s Theorem}
\author{Hisashi Aratake%
	\thanks{\textit{E-mail address}: \texttt{aratake@kurims.kyoto-u.ac.jp}\\
		}
	}
\affil{Research Institute for Mathematical Sciences,\\
	Kyoto University, Kyoto, Japan}
\date{}

\begin{document}
\maketitle

\begin{abstract}
	Sheaves of structures are useful to give constructions in universal algebra and model theory.
	We can describe their logical behavior in terms of Heyting-valued structures.
	In this paper, we first provide a systematic treatment of sheaves of structures and Heyting-valued structures from the viewpoint of categorical logic.
	We then prove a form of {\L}o\'{s}'s theorem for Heyting-valued structures.
	We also give a characterization of Heyting-valued structures for which {\L}o\'{s}'s theorem holds with respect to any maximal filter.
\end{abstract}

\setcounter{section}{-1}

\section{Introduction}

Sheaf-theoretic constructions have been used in universal algebra and model theory.
In this context, sheaves of abelian groups or rings in geometry are generalized to \textit{sheaves of structures}.
We can obtain, for example, the product (resp.\ an ultraproduct) of a family of structures from some sheaf by taking the set of global sections (resp.\ a stalk).
This viewpoint originated from the early literature \cite{Comer1974}, \cite{Ell1974} and \cite{Mac1973}.
In combination with the theory of sheaf representations of algebras, 
Macintyre \cite{Mac1973} succeeded in giving model-companions of some theories of commutative rings
by transferring model-theoretic properties from stalks to global sections.

On the other hand, sheaves have another description as \textit{Heyting-valued sets}.
The notion of Heyting-valued sets originally arises from that of Boolean-valued models of set theory, which was introduced in relation to Cohen's forcing.
The development of topos theory in the early seventies, mainly due to Lawvere \& Tierney, revealed profound relationships between toposes and models of set theory;
objects in a topos can be regarded as ``generalized sets'' in a universe.
Subsequently, Fourman \& Scott \cite{FourSco1979} and Higgs 
\footnote{Originally in his preprint written in 1973, part of which was later published as \cite{Higgs1984}.}
independently established the categorical treatment of Heyting-valued sets (see \textbf{Remark \ref{rmk:relation-to-set-theory}}).
The category $ \mathbf{Sh}(X) $ of sheaves of sets on a space $ X $ and the category $ \mathbf{Set}(\mathcal{O}(X)) $ of $ \mathcal{O}(X) $-valued sets 
turned out to be categorically equivalent.

Some model-theorists of that era immediately applied Heyting-valued sets to concrete problems in sheaf-theoretic model theory.
However, general methods of Heyting-valued model theory have not been explored enough,
though Fourman \& Scott mentioned such a direction in the preamble of \cite{FourSco1979}.
Even worse, we are not aware of any clear explanation of the relationship between sheaves of structures and Heyting-valued structures.
In this paper, employing well-established languages of categorical logic,
we will give a gentle and coherent account of Heyting-valued semantics of first-order logic from the categorical point of view,
and will apply that framework to obtain a generalization of {\L}o\'{s}'s theorem for Heyting-valued structures.
We also provide a characterization of Heyting-valued structures for which {\L}o\'{s}'s theorem holds w.r.t.\ any maximal filter.
Our theorems improve the works of Caicedo \cite{Cai1995} and Pierobon \& Viale \cite{PieVia2020}.
While our principal examples of Heyting-valued sets are sheaves on topological spaces,
other natural examples include sheaves on the complete Boolean algebra of regular open sets and Boolean-valued sets on the measure algebra.
Therefore, we will develop our theory based on any complete Heyting algebra (a.k.a.\ a frame or a locale) rather than a topological space.

The intended audience for this paper is anyone who has interests both in model theory and in categorical logic.
We assume some familiarity with topos theory and first-order categorical logic.
Most categorical prerequisites are covered by \cite{SGL}.
In \S\ref{subsec:structures-in-topos}, we will recall some elements of first-order categorical logic. 

The areas related to this paper (and its sequels in the future) are diverse,
including model theory, universal algebra, set theory, categorical logic, topos theory, and their applications to ordinary mathematics.
The author did his best to ensure that the reader can follow the scattered literature (especially in model theory and topos theory) on each occasion during the course.

\paragraph{The structure of this paper:}
In \S\ref{sec:preliminaries}, we will begin with preliminaries on sheaves and Heyting-valued sets.
After we see basic properties of Heyting-valued sets and morphisms between them, we will give an outline of the equivalence of sheaves and Heyting-valued sets.
We also provide some details on the topos $ \mathbf{Set}(\mathcal{O}(X)) $ of $ \mathcal{O}(X) $-valued sets.
In \S\ref{sec:sheaf-of-str-and-H-valued-str}, we will study structures in the toposes $ \mathbf{Sh}(X) $ and $ \mathbf{Set}(\mathcal{O}(X)) $
and the relationship between them.
We will also introduce forcing values of formulas categorically.
In \S\ref{sec:filt-quot-and-Los-thm}, observing that sheaves of structures generalize some model-theoretic constructions,
we will introduce a further generalization of filter-quotients of sheaves to Heyting-valued structures 
and prove {\L}o\'{s}'s theorem and the characterization theorem.
In \S\ref{sec:conclusion}, we will indicate some possible future directions with an expanded list of previous works.

\paragraph*{}
Closing the introduction, we have to mention Loullis' work \cite{Lou1979} on Boolean-valued model theory.
The starting point of this research was trying to digest his work from a modern categorical viewpoint,
though our work is still too immature to give the reader a full explanation of his contribution.
If the author had not met his work, this paper would not have existed.
The author regrets his untimely death, according to \cite{BunRey1981}, in 1978.

\paragraph{Acknowledgment:}
The author is grateful to his supervisor, Kazushige Terui, for careful reading of an earlier draft and many helpful suggestions on the presentation of this paper.
He also thanks Soichiro Fujii for useful comments.

\section{Sheaves and Heyting-Valued Sets} \label{sec:preliminaries}
Heyting-valued sets were introduced independently by Higgs \cite{Higgs1984} and by Fourman \& Scott \cite{FourSco1979}.
In this section, we will review the construction of the category $ \mathbf{Set}(\mathcal{O}(X)) $ of $ \mathcal{O}(X) $-valued sets for a locale $ X $,
its relation to sheaves on $ X $, and its categorical structures as a topos.
Most results are covered by \cite{Higgs1984}, \cite{FourSco1979}, \cite[\S C1.3]{Elephant} and \cite[Chapter 2]{HoCA3}.
For the reader's convenience, we will occasionally give brief sketches of proofs.

For aspects of Heyting-valued sets in intuitionistic logic, see \cite[Chapters 13--14]{TvD}.
As a category, $ \mathbf{Set}(\mathcal{O}(X)) $ is a prototypical example of the topos obtained from a tripos
(see \cite{HJP1980} and \cite[Chapter 2]{vOosten2008}).
The internal logic of $ \mathbf{Set}(\mathcal{O}(X)) $ is reduced to the logic of tripos.
Walters \cite{Walters1981}, \cite{Walters1982} developed another direction of generalization of Heyting-valued sets.

\subsection{Heyting-Valued Sets} \label{subsec:H-valued-set}

\begin{Def}
	A \emph{frame} is a complete lattice satisfying the infinitary distributive law:
	\[ a \wedge \bigvee_i b_i = \bigvee_i a \wedge b_i. \]
	In particular, any frame has $ 0 $ and $ 1 $.
\end{Def}

A frame is the same thing as a complete Heyting algebra:
the infinitary distributive law for a frame $ H $ says that each monotone map $ a \wedge (-) \colon H \to H $ preserves arbitrary joins.
This happens exactly when each map $ a  \wedge (-) $ has a right adjoint $ a \imply (-) \colon H \to H $, i.e., a monotone map satisfying
\[ \forall b,c \in H,\, [a \wedge b \leq c \iff b \leq a \limply c ]. \]
This fact follows either from category theory (the \textit{General Adjoint Functor Theorem}), or from a direct construction
\[ a \imply c := \bigvee \Set{b \in H \smid a \wedge b \leq c}. \]

On the other hand, frame homomorphisms differ from those for complete Heyting algebras (and even those for complete lattices):
\begin{Def}
	Let $ H,H' $ be frames.
	A \emph{frame homomorphism} $ h \colon H \to H' $ is a map from $ H $ to $ H' $ preserving finite meets and arbitrary joins.
	Let $ \mathbf{Frm} $ denote the category of frames.
\end{Def}

Similarly to the above, any frame homomorphism $ h \colon H \to H' $ has a right adjoint $ k \colon H' \to H $ given by
\[ k(b) = \bigvee \Set{a \in H \smid h(a) \leq b}. \]

Any continuous map $ f \colon X \to Y $ of topological spaces gives rise to a frame homomorphism $ f^* \colon \mathcal{O}(Y) \to \mathcal{O}(X) $ given by $ f^*(V)=f^{-1}(V) $,
where $ \mathcal{O}(X) $ (resp.~$ \mathcal{O}(Y) $) is the frame of open sets of $ X $ (resp.~$ Y $).
The functor $ \mathcal{O}(-) \colon \mathbf{Top}^{\op} \to \mathbf{Frm} $ is full and faithful on sober spaces
(i.e., spaces satisfying a suitable axiom between $ T_0 $ and $ T_2 $).
Therefore, it translates the language of spaces to that of frames, and we may consider frames as ``point-free'' spaces.
This justifies the following definition:

\begin{Def} \label{def:cat-of-locales}
	A frame considered as an object of $ \mathbf{Frm}^{\op } $ is called a \emph{locale}.
	We denote $ \mathbf{Frm}^{\op } $ by $ \mathbf{Loc} $ and the frame corresponding to a locale $ X \in \mathbf{Loc}$ by $ \mathcal{O}(X) $.
	We will write $ U,V, $ etc.~for elements of $ \mathcal{O}(X) $ and $ 0_X $ (resp.~$ 1_X $) for the smallest (resp.~largest) element.
	
	For a morphism $ f \colon X \to Y $ of locales, the corresponding frame homomorphism is denoted by $ f^* \colon \mathcal{O}(Y) \to \mathcal{O}(X) $.
	$ f^* $ has a right adjoint $ f_* \colon \mathcal{O}(X) \to \mathcal{O}(Y) $.
	Morphisms of locales are also called continuous maps of locales.
\end{Def}

By writing $ X_{\ell} $ for the locale given by a topological space $ X $, i.e., $ \mathcal{O}(X_{\ell}) = \mathcal{O}(X) $,
we now have a functor $ (-)_{\ell} \colon \mathbf{Top} \to \mathbf{Loc} $.
It has a right adjoint $ \mathrm{pt} \colon \mathbf{Loc} \to \mathbf{Top} $ sending a locale $ X $ 
to the space $ \mathrm{pt}(X) $ of ``points of $ X $'' (see, e.g., \cite[Chapter IX]{SGL}).
For more on frames and locales in point-free topology, see \cite{JohSS} and \cite{PicPulFL}.

We are now ready to define Heyting-valued sets.
In the remainder of this section, we fix a locale $ X $.
\begin{Def}
	An \emph{$ \mathcal{O}(X) $-valued set} $ (A,\alpha) $ is a pair of a set $ A $ and a map $ \alpha \colon A \times A \to \mathcal{O}(X) $ such that
	\begin{itemize}
		\item $ \forall a,b \in A,\, \alpha(a,b) = \alpha(b,a) $
		\item $ \forall a,b,c \in A,\, \alpha(a,b) \land \alpha(b,c) \leq \alpha(a,c) $ \qedhere
	\end{itemize}
\end{Def}

In the logic of tripos (associated with $ X $), $ \alpha $ is a ``partial equivalence relation'' on $ A $.
Instead of $ \alpha(a,b) $, the notation $ \dbracket{a=b} $ is frequently used in the literature.
We introduce a few conventions:
\begin{itemize}
	\item $ \alpha(a) := \alpha(a,a) $ is called the \emph{extent} of $ a $. Note that $ \alpha(a,b) \leq \alpha(a) \land \alpha(b) $.
	
	\item If $ \alpha(a) = 1_X $, then $ a $ is called a \emph{global element} of $ (A,\alpha) $.
\end{itemize}

Morphisms of Heyting-valued sets should be ``functional relations'' (again in the logic of tripos).

\begin{Def}
	Let $ (A,\alpha) , (B,\beta) $ be $ \mathcal{O}(X) $-valued sets.
	A \emph{morphism $ \varphi \colon (A,\alpha) \to (B,\beta) $ of $ \mathcal{O}(X) $-valued sets} is a map $ A \times B \to \mathcal{O}(X) $ which satisfies
	\begin{align*}
		\forall a,a' \in A,\, \forall b,b' \in B,\, & \quad \alpha(a,a') \wedge \varphi(a,b) \wedge \beta(b,b') \leq \varphi(a',b'), \\
		\forall a \in A,\, \forall b,b' \in B,\, & \quad \varphi(a,b) \wedge \varphi(a,b') \leq \beta(b,b'), \\
		\forall a \in A,\, & \quad \alpha(a) = \bigvee_{b \in B} \varphi(a,b).
	\end{align*}
	In particular, $ \varphi(a,b) \leq \alpha(a) \wedge \beta(b) $ always holds.
	If $ \psi \colon (B,\beta) \to (C,\gamma) $ is another morphism, we can define the composite $ \psi \circ \varphi $ by
	\[ (\psi \circ \varphi)(a,c) = \bigvee_{b \in B} \varphi(a,b) \wedge \psi(b,c). \]
	We write $ \mathbf{Set}(\mathcal{O}(X)) $ for the category of $ \mathcal{O}(X) $-valued sets and morphisms,
	where the identity $ \id_{(A,\alpha)} $ is given by $ \alpha $ itself.
\end{Def}

\begin{Rmk} \label{rmk:relation-to-set-theory}
	In set theory, for a frame $ H $, we can construct a model $ V^{(H)} $ of intuitionistic set theory (\cite[Chapter IV]{BellIST}). 
	$ V^{(H)} $ is called the \emph{Heyting-valued universe}.
	The category $ \mathbf{Set}(H) $ is regarded as a categorical counterpart of $ V^{(H)} $ (\cite[Appendix]{BellBVMIP}, \cite{ACM2019}),
	and we can take arguments and examples from set theory to investigate Heyting-valued sets (cf.~\cite{PieVia2020}).
\end{Rmk}

We list useful facts on Heyting-valued sets, some of which will not be used in this paper.
\begin{Lem} \label{lem:when-are-two-morphisms-identical}
	Two morphisms $ \varphi,\psi \colon (A,\alpha) \rightrightarrows (B,\beta) $ are identical if
	\[ \forall a \in A,\, \forall b \in B ,\, \varphi(a,b) \leq \psi(a,b). \]
\end{Lem}
\begin{proof}
	Suppose $ \forall a \in A,\, \forall b \in B ,\, \varphi(a,b) \leq \psi(a,b) $. Then,
	\begin{align*}
		\psi(a,b) & = \psi(a,b) \wedge \alpha(a) = \psi(a,b) \wedge \bigvee_{b'} \varphi(a,b')  \\
		& \leq \bigvee_{b'}[\varphi(a,b') \wedge \psi(a,b') \wedge \psi(a,b) ] \leq \bigvee_{b'} [\varphi(a,b') \wedge \beta(b',b)] = \varphi(a,b).
	\end{align*}
\end{proof}

\begin{Prop*} \label{prop:monic-and-epic-morphisms}
	Let $ \varphi \colon (A,\alpha) \to (B,\beta) $ be a morphism in $ \mathbf{Set}(\mathcal{O}(X)) $.
	\begin{enumerate}
		\item $ \varphi $ is a monomorphism if and only if
		\[ \forall a,a' \in A,\, \forall b \in B,\, \varphi(a,b) \wedge \varphi(a',b) \leq \alpha(a,a'). \]
		
		\item $ \varphi $ is an epimorphism if and only if
		\[ \forall b \in B,\, \beta(b) = \bigvee_{a \in A} \varphi(a,b). \]
		
		\item $ \varphi $ is an isomorphism if and only if it is monic and epic.
		In other words, $ \mathbf{Set}(\mathcal{O}(X)) $ is a balanced category. 
		If $ \varphi $ is an isomorphism, $ \varphi^{-1} $ is given by $ \varphi^{-1}(b,a) = \varphi(a,b) $. \qedhere
	\end{enumerate}
\end{Prop*}

\begin{Def}
	We say that a morphism $ \varphi \colon (A,\alpha) \to (B,\beta) $ is \emph{represented by a map} $ h \colon A \to B $ when
	\[ \forall a \in A,\, \forall b  \in B,\, \varphi(a,b) = \alpha(a) \wedge \beta(ha,b). \qedhere \]
\end{Def}

\begin{Prop*} \label{prop:representable-morphisms}
	\ 
	
	\begin{enumerate}
		\item A morphism $ \varphi \colon (A,\alpha) \to (B,\beta) $ is represented by a map $ h \colon A \to B $ if and only if 
		\[ \forall a \in A,\, \forall b  \in B,\, \varphi(a,b) \leq \beta(ha,b). \]
		
		\item A map $ h \colon A \to B $ represents some morphism from $ (A,\alpha) $ to $ (B,\beta) $ if and only if 
		\[ \forall a, a' \in A,\, \alpha(a,a') \leq \beta(ha,ha'). \]
		Moreover, if $ h $ further satisfies $ \alpha(a) = \beta(ha) $ for all $ a \in A $, 
		then the morphism $ \varphi $ represented by $ h $ is given simply by $ \varphi(a,b) = \beta(ha,b) $.
		
		\item Suppose two maps $ h, k \colon A \to B $ represent some morphisms.
		They represent the same morphism if and only if
		\[ \forall a \in A,\, \alpha(a) \leq \beta(ha,ka). \]
		
		\item Let $ (C,\gamma) $ be another $ \mathcal{O}(X) $-valued set and $ \psi \colon (B,\beta) \to (C,\gamma) $ be another morphism.
		If $ \varphi $ (resp.~$ \psi $) is represented by a map $ h $ (resp.~$ k $), then $ \psi\varphi $ is represented by $ kh $. \qedhere
	\end{enumerate}
\end{Prop*}

\begin{Prop*} \label{prop:monic-and-epic-representables}
	Let $ \varphi \colon (A,\alpha) \to (B,\beta) $ be a morphism represented by $ h $.
	\begin{itemize}
		\item $ \varphi $ is monic $ \iff \forall a,a' \in A,\, \alpha(a,a') = \alpha(a) \wedge \alpha(a') \wedge \beta(ha,ha') $.
		
		\item $ \varphi $ is epic $ \iff \forall b \in B,\, \beta(b) = \bigvee_a [\alpha(a) \wedge \beta(ha,b)] $.
	\end{itemize}
	Further, if $ \alpha(a) = \beta(ha) $ for all $ a \in A $, these conditions reduce to
	\begin{itemize}
		\item $ \varphi $ is monic $ \iff \forall a,a' \in A,\, \alpha(a,a') = \beta(ha,ha') $.
		
		\item $ \varphi $ is epic $ \iff \forall b \in B,\, \beta(b) = \bigvee_a \beta(ha,b) $. \qedhere
	\end{itemize}
\end{Prop*}

Combining the above facts, we obtain
\begin{Cor*} \label{cor:isomorphisms-induced-by-maps}
	If $ h $ satisfies $ \alpha(a) = \beta(ha) $ for all $ a \in A $,
	then $ h $ represents an isomorphism exactly when the following conditions hold
	\[ \forall a,a' \in A,\, \alpha(a,a') = \beta(ha,ha'), \; \text{and} \; \forall b \in B,\,  \beta(b) = \bigvee_a \beta(ha,b). \]
	If $ h $ satisfies these conditions, the induced isomorphism $ \varphi(a,b)=\beta(ha,b) $ has the inverse $ \varphi^{-1}(b,a)=\beta(ha,b) $.
\end{Cor*}

\subsection{Sheaves on Locales and Complete Heyting-Valued Sets} \label{subsec:sheaves-vs-OX-sets}
Continuing from the previous section, we fix a locale $ X $.
Let us discuss the relationship between sheaves on $ X $ and $ \mathcal{O}(X) $-valued sets.
Our presentation style here is largely due to \cite[\S C1.3]{Elephant}.

\begin{Def}
	\ 
	
	\begin{enumerate}
		\item A \emph{presheaf} on $ X $ is a functor $ \mathcal{O}(X)^{\op} \to \mathbf{Set} $.
		For a presheaf $ P $ and $ U \in \mathcal{O}(X) $, elements of $ PU $ (resp.\ of $ P1_X $) are called \emph{sections} of $ P $ on $ U $ (resp.\ \emph{global sections} of $ P $).
		If $ a \in PU $ and $ W \leq U $, we will write $ a|_W $ for $ P(W \leq U)(a) $.
		
		\item A presheaf $ P $ is said to be a \emph{sheaf} on $ X $ when it satisfies the following condition:
		For any covering $ \{U_i\}_i $ of $ U \in \mathcal{O}(X) $ (i.e., $ U = \bigvee_i U_i $) and any family $ \{a_i\}_i $ of sections $ a_i \in PU_i $,
		if $ a_i|_{U_i \wedge U_j} = a_j|_{U_i \wedge U_j} $ for all $ i,j $, then there exists a unique $ a \in PU $ such that $ a|_{U_i} =a_i $ for all $ i $.
		
		\item \emph{Morphisms of presheaves} are defined to be natural transformations.
		The functor category $ \mathbf{Set}^{\mathcal{O}(X)^{\op}} $ is also called the category of presheaves.
		Let $ \mathbf{Sh}(X) $ denote its full subcategory spanned by sheaves. \qedhere
	\end{enumerate}
\end{Def}

We can associate a presheaf $ P $ with an $ \mathcal{O}(X) $-valued set $ \Theta(P) := (\coprod_U PU,\delta_P) $ as follows:
for $ (a,b) \in PU \times PV \subseteq \coprod_U PU \times \coprod_U PU $,
\[ \delta_P(a,b) := \bigvee \Set{W \leq U \land V \smid a|_W = b|_W}. \]

Notice that
\begin{itemize}
	\item $ a \in PU $ if and only if $ \delta_P(a) = U $.
	Hence, global elements of $ \Theta(P) $ are exactly global sections of $ P $.
	\item If $ P $ is a sheaf, then $ \delta_P(a,b) $ is the largest element on which the restrictions of $ a $ and $ b $ coincide.
	Moreover, if $ X $ is a topological space,
	\[ \delta_P(a,b) = \Set{x \in X \smid a_x = b_x }, \]
	where $ a_x,b_x $ are the germs of $ a,b $ over $ x $.
\end{itemize}

For a morphism $ \xi \colon P \to Q $ of presheaves on $ X $, the induced map $ h \colon \coprod_U PU \to \coprod_{U} QU $ satisfies
\[ \forall (a,b) \in \biggl  (\coprod_U PU \biggr )^2,\, \delta_P(a,b) \leq \delta_Q(ha,hb) \; \text{and}\; \delta_P(a) = \delta_Q(ha). \]
Therefore, by \textbf{Proposition \ref{prop:representable-morphisms}}, $ h $ represents a morphism $ \Theta(\xi) \colon \Theta(P) \to \Theta(Q) $.
This construction gives a functor $ \Theta \colon \mathbf{Set}^{\mathcal{O}(X)^{\op}} \to \mathbf{Set}(\mathcal{O}(X)) $.

Notice that a presheaf $ P $ is separated if and only if, for any $a,b$, $ \delta_P(a)=\delta_P(b)=\delta_P(a,b) $ implies $ a = b $.
Fourman \& Scott \cite{FourSco1979} say Heyting-valued sets are separated if the latter condition holds.
To give a similar characterization of sheaves, we need a more involved definition.

\begin{Def}
	For an $ \mathcal{O}(X) $-valued set $ (A,\alpha) $, define a preorder $ \sqsubseteq $ on $ A $ by
	\[ b \sqsubseteq a \defarrow \alpha(a,b) = \alpha(b). \]
	$ (A,\alpha) $ is said to be \emph{complete} if the following conditions hold:
	\begin{itemize}
		\item $ \sqsubseteq $ is a partial order. (This is equivalent to separatedness.)
		
		\item For any $ a \in A $ and $ U \leq \alpha(a) $, there exists $ b \in A $ such that $ b \sqsubseteq a $ and $ \alpha(b) = U $.
		(If $ \sqsubseteq $ is a partial order, such $ b $ is uniquely determined and denoted by $ a|_U $.)
		
		\item If a family $ \{a_i\}_i $ of elements of $ A $ is pairwise compatible, i.e., $ \alpha(a_i,a_j) = \alpha(a_i) \wedge \alpha(a_j) $ for all $ i,j $,
		then it has a supremum w.r.t.\ $ \sqsubseteq $. (The supremum is called an amalgamation of $ \{a_i\}_i $.)
	\end{itemize}
	Let $ \mathbf{CSet}(\mathcal{O}(X)) $ denote the full subcategory of $ \mathbf{Set}(\mathcal{O}(X)) $ spanned by complete $ \mathcal{O}(X) $-valued sets.
\end{Def}

\begin{Prop}
	A presheaf $ P $ on $ X $ is a sheaf if and only if $ \Theta(P) $ is complete as an $ \mathcal{O}(X) $-valued set. 
	Moreover, for any complete $ \mathcal{O}(X) $-valued set $ (A,\alpha) $,
	there exists a sheaf $ P $ on $ X $ such that $ (A,\alpha) $ and $ \Theta(P) $ are isomorphic.
\end{Prop}
\begin{proof}
	The latter part:
	if $ (A,\alpha) $ is complete, then by putting 
	\[ PU := \Set{a \in A \smid \alpha(a) = U}, \]
	we have a desired sheaf $ P $ on $ X $.
\end{proof}

We can rephrase completeness in terms of singletons.

\begin{Def}
	Let $ (A,\alpha) $ be an $ \mathcal{O}(X) $-valued set.
	A \emph{singleton} on $ (A,\alpha) $ is a function $ \sigma \colon A \to \mathcal{O}(X) $ such that
	\[ \forall a,a' \in A,\, \sigma(a) \wedge \alpha(a,a') \leq \sigma(a') \; \text{and}\; \sigma(a) \wedge \sigma(a') \leq \alpha(a,a'). \]
	In particular, $ \sigma(a) \leq \alpha(a) $ always holds.
\end{Def}

For each $ a \in A $, the map $ \sigma_a:=\alpha(a,-) $ is a singleton of $ (A,\alpha) $.

\begin{Lem}
	For an $ \mathcal{O}(X) $-valued set $ (A,\alpha) $, TFAE:
	\begin{enumerate}[label=(\roman*)]
		\item  $ (A,\alpha) $ is complete.
		\item Any singleton of $ (A,\alpha) $ is of the form $ \sigma_a $ for a uniquely determined $ a $.
	\end{enumerate}
\end{Lem}
\begin{proof}
	(i)$ \Rightarrow $(ii):
	Suppose $ (A,\alpha) $ is complete. Let $ \sigma $ be a singleton on $ (A,\alpha) $.
	Then the family $ \Set{a|_{\sigma(a)} \smid a \in A} $ is pairwise compatible, and its supremum $ s \in A $ satisfies $ \sigma=\sigma_s $.
	
	\noindent (ii)$ \Rightarrow $(i):
	Suppose the condition (ii) holds. If $ \alpha(a)=\alpha(a')=\alpha(a,a') $, then $ \sigma_a=\sigma_{a'} $ and hence $ a=a' $.
	Thus $ \sqsubseteq $ is anti-symmetric.
	If $ a \in A $ and $ U \leq \alpha(a) $, the map $ \alpha(a,-) \wedge U $ is a singleton, and  we then have the restriction $ a|_{U} $.
	If a family $ \{a_i\}_i $ is pairwise compatible, the map $ \bigvee_i \alpha(a_i,-) $ is a singleton, and we then have the amalgamation.
\end{proof}

\begin{Lem} \label{lem:morphism-to-complete-set}
	Let $ (A,\alpha), (B,\beta) $ be $ \mathcal{O}(X) $-valued sets with $ (B,\beta) $ complete.
	Each morphism $ \varphi \colon (A,\alpha) \to (B,\beta) $ is represented by 
	a unique map $ h \colon A \to B $ which satisfies $ \alpha(a,a') \leq \beta(ha,ha') $ and $ \alpha(a) = \beta(ha) $ for all $ a,a' \in A $.
\end{Lem}
\begin{proof}
	For any fixed $ a \in A $, the map $ \varphi(a,-) $ is a singleton of $ (B,\beta) $.
	By completeness, we can find a unique $ ha \in B $ such that $ \varphi(a,b) = \beta(ha,b) $ for every $ b \in B$.
	This defines a map $ h \colon A \to B $ representing $ \varphi $ and having the desired properties.
\end{proof}

This lemma and \textbf{Proposition \ref{prop:representable-morphisms}}(4) yield
\begin{Prop*}
	$ \Theta $ induces a categorical equivalence between $ \mathbf{Sh}(X) $ and $ \mathbf{CSet}(\mathcal{O}(X)) $.
\end{Prop*}

On the other hand, we can also show that $ \mathbf{Set}(\mathcal{O}(X)) $ and $ \mathbf{CSet}(\mathcal{O}(X)) $ are categorically equivalent.

\begin{Prop*} \label{prop:completion-of-OX-sets}
	Let $ \tilde{A} $ be the set of singletons on $ (A,\alpha) $. Define a valuation $ \tilde{\alpha} $ on $ \tilde{A} $ by
	\[ \tilde{\alpha}(\sigma,\tau) := \bigvee_{a \in A} \sigma(a) \wedge \tau(a). \]
	Then $ (\tilde{A},\tilde{\alpha}) $ is a complete $ \mathcal{O}(X) $-valued set, and
	the map $ \tilde{A} \times A \ni (\sigma,a) \mapsto \sigma(a) \in \mathcal{O}(X) $ represents an isomorphism $ (A,\alpha) \simeq (\tilde{A},\tilde{\alpha}) $.
	We call $ (\tilde{A},\tilde{\alpha}) $ the \emph{completion of $ (A,\alpha) $}.
	
	For a morphism $ \varphi \colon (A,\alpha) \to (B,\beta) $, let $ \tilde{\varphi} $ be the composite of
	\[ (\tilde{A},\tilde{\alpha}) \isoarrow (A,\alpha) \xrightarrow{\varphi} (B,\beta) \isoarrow (\tilde{B},\tilde{\beta}). \]
	Then $ \tilde{(-)} \colon \mathbf{Set}(\mathcal{O}(X)) \to \mathbf{CSet}(\mathcal{O}(X)) $ becomes a functor and
	also gives a quasi-inverse of the inclusion functor.
\end{Prop*}

\begin{Cor*} \label{cor:Sh(X)=Set(OX)}
	The categories $ \mathbf{Sh}(X) $, $ \mathbf{Set}(\mathcal{O}(X)) $ and $ \mathbf{CSet}(\mathcal{O}(X)) $ are categorically equivalent.
	In particular, $ \mathbf{Set}(\mathcal{O}(X)) $ is a Grothendieck topos.
\end{Cor*}

\subsection{Topos Structure of $ \mathbf{Set}(\mathcal{O}(X)) $} \label{subsec:topos-structure-of-Set(OX)}
In the previous section, we saw that $ \mathbf{Set}(\mathcal{O}(X)) $ is a Grothendieck topos.
Here, we will give a concrete description of the topos structure of $ \mathbf{Set}(\mathcal{O}(X)) $.
Most results here (except for some details on the lattice $ \mathcal{P}(A,\alpha) $) are borrowed from \cite{Higgs1984}.
The constructions will be exploited later in this paper.

\begin{Prop*}[Finite limits in $ \mathbf{Set}(\mathcal{O}(X))$] \label{prop:fin-limits-in-Set(OX)}
	\ 
	
	\begin{enumerate}
		\item Let $ (\{\ast\},\top) $ be the $ \mathcal{O}(X) $-valued set with $ \top(\ast,\ast) = 1_X $.
		This yields a terminal object in $ \mathbf{Set}(\mathcal{O}(X)) $.
		
		\item Let $ \{(A_i,\alpha_i)\}_{i \in I} $ be a finite family of $ \mathcal{O}(X) $-valued sets.
		Define a valuation $ \delta $ on $ \prod_i A_i $ by $ \delta(a,a') = \bigwedge_i \alpha_i(a_i,a'_i) $ for $ a = \{a_i\}_{i \in I} $ and $ a' = \{a'_i\}_{i \in I} $.
		Then $ (\prod_i A_i,\delta) $ equipped with the canonical projections $ (\prod_i A_i,\delta) \to (A_i,\alpha_i) $
		is a product of $ \{(A_i,\alpha_i)\}_{i \in I} $ in $ \mathbf{Set}(\mathcal{O}(X)) $.
		
		\item Let $ \varphi, \psi \colon (A,\alpha) \rightrightarrows (B,\beta) $ be morphisms of $ \mathcal{O}(X) $-valued sets.
		Define a valuation $ \delta $ on $ A $ by $ \delta(a,a') = \alpha(a,a') \wedge \bigvee_{b \in B} \varphi(a,b) \wedge \psi(a,b) $.
		Then $ (A,\delta) $ equipped with the canonical morphism $ (A,\delta) \rightarrowtail (A,\alpha) $
		is an equalizer of $ \varphi $ and $ \psi $ in $ \mathbf{Set}(\mathcal{O}(X)) $.
		
		\item Let $ (A,\alpha) \xrightarrow{\varphi} (C,\gamma) \xleftarrow{\psi} (B,\beta) $ be morphisms of $ \mathcal{O}(X) $-valued sets.
		Define a valuation $ \delta $ on $ A \times B $ by 
		\[ \delta((a,b),(a',b')) = \alpha(a,a') \wedge \beta(b,b') \wedge \bigvee_{c \in C} \varphi(a,c) \wedge \psi(b,c). \]
		Then $ (A \times B,\delta) $ equipped with the canonical projections is a pullback of that diagram in $ \mathbf{Set}(\mathcal{O}(X)) $. \qedhere
	\end{enumerate}
\end{Prop*}

The following notion of strict relation is crucial in handling subobjects of an $ \mathcal{O}(X) $-valued set.
In the next section, it will enable us to define the ``forcing values'' of formulas.

\begin{Def}
	Let $ (A,\alpha) $ be an $ \mathcal{O}(X) $-valued set.
	A \emph{strict relation} on $ (A,\alpha) $ is a function $ \sigma \colon A \to \mathcal{O}(X) $ such that
	\[ \forall a,a' \in A,\, \sigma(a) \wedge \alpha(a,a') \leq \sigma(a') \; \text{and}\; \sigma(a) \leq \alpha(a). \]
	Note that a singleton is a strict relation on the same $ \mathcal{O}(X) $-valued set.
\end{Def}

\begin{Prop} \label{prop:strict-relations-vs-subobjects}
	Let $ \mathcal{P}(A,\alpha) $ be the set of strict relations on $ (A,\alpha) $ ordered by
	\[ \sigma \leq \tau \defarrow \forall a\in A,\, \sigma(a) \leq \tau(a). \]
	Then, as ordered sets, $ \mathcal{P}(A,\alpha) $ is isomorphic to the poset $ \Sub(A,\alpha) $ of subobjects of $ (A,\alpha) $.
\end{Prop}
\begin{proof}
	For a strict relation $ \sigma $, we define a map $ \alpha_{\sigma} \colon A \times A \to \mathcal{O}(X) $ by
	\[ \alpha_{\sigma}(a,a') := \sigma(a) \wedge \alpha(a,a') = \sigma(a') \wedge \alpha(a,a'). \]
	Then $ (A,\alpha_{\sigma}) $ is an $ \mathcal{O}(X) $-valued set and the identity map on $ A $ 
	represents a monomorphism $ \iota_{\sigma} \colon (A,\alpha_{\sigma}) \rightarrowtail (A,\alpha) $ by \textbf{Proposition \ref{prop:monic-and-epic-representables}}.
	
	Conversely, for a monomorphism $ \varphi \colon (B,\beta) \rightarrowtail (A,\alpha) $,
	we define a strict relation $ \rho_{\varphi} \colon A \to \mathcal{O}(X) $ by
	\[ \rho_{\varphi}(a) := \bigvee_{b \in B} \varphi(b,a). \]
	Then we can check 
	\begin{itemize}
		\item for any $ \sigma $, $ \rho_{(\iota_{\sigma})} = \sigma $.
		\item for any $ \varphi $, $ \iota_{(\rho_{\varphi})} \simeq \varphi $ as subobjects of $ (A,\alpha) $. \qedhere
	\end{itemize}
\end{proof}

In fact, for an arbitrary morphism $ \varphi $, a strict relation $ \rho_{\varphi} $ can be defined as above,
and the image factorization of $ \varphi $ is given by
\[ \tikz[auto]{
	\node (UL) at (-2,1.5) {$ (B,\beta) $};
	\node (UR) at (2,1.5) {$ (A,\alpha) $};
	\node (DC) at (0,0) {$ (A,\alpha_{(\rho_{\varphi})}) $};
	\draw[->] (UL) to node {$ \varphi $} (UR);
	\draw[->>] (UL) to node[swap] {$ \varphi $} (DC);
	\draw[>->] (DC) to node[swap] {$ \iota_{(\rho_{\varphi})} $} (UR);
}. \]

$ \mathbf{Set}(\mathcal{O}(X)) $ is a topos and, in particular, a Heyting category (\cite[\S A1.4]{Elephant}).
The associated operations on subobject lattices are as follows:

\begin{Prop*}
	The operations on the frame $ \mathcal{P}(A,\alpha) $ are given by
	\begin{gather*}
		1_{\mathcal{P}(A,\alpha)}(a) = \alpha(a), \qquad 0_{\mathcal{P}(A,\alpha)}(a) = 0_X, \\
		(\sigma \wedge \tau)(a) = \sigma(a) \wedge \tau(a), \quad \Bigl ( \bigvee_i \sigma_i \Bigr )(a) = \bigvee_i \sigma_i(a), \\
		(\sigma \imply \tau)(a) = \alpha(a) \wedge (\sigma(a) \imply \tau(a)). \qedhere
	\end{gather*}
\end{Prop*}

\begin{Prop}
	Let $ \varphi \colon (B,\beta) \to (A,\alpha) $ be a morphism.
	Pulling back subobjects along $ \varphi $ defines a frame homomorphism $ \varphi^* \colon \mathcal{P}(A,\alpha) \to \mathcal{P}(B,\beta) $
	such that, for $ \sigma \in \mathcal{P}(A,\alpha) $,
	\[ (\varphi^*\sigma)(b) = \bigvee_{a \in A} \varphi(b,a) \wedge \sigma(a) = \beta(b) \wedge \bigwedge_{a \in A} [\varphi(b,a) \imply \sigma(a)]. \]
\end{Prop}
\begin{proof}
	For the last identity, note that
	\begin{align*}
		\beta(b) \wedge \bigwedge_{a \in A} [\varphi(b,a) \imply \sigma(a)] & = \bigvee_{a \in A} \varphi(b,a) \wedge \bigwedge_{a' \in A} [\varphi(b,a') \imply \sigma(a')] \\
		& \leq \bigvee_{a \in A} \varphi(b,a) \wedge [\varphi(b,a) \imply \sigma(a)] \\
		& \leq \bigvee_{a \in A} \varphi(b,a) \wedge \sigma(a). \qedhere
	\end{align*}
\end{proof}

\begin{Prop*}
	In the same notations as above,
	$ \varphi^* $ has both a left adjoint $ \exists_{\varphi} $ and a right adjoint $ \forall_{\varphi} $: for $ \tau \in \mathcal{P}(B,\beta) $,
	\begin{align*}
		(\exists_{\varphi}\tau)(a) & = \bigvee_{b \in B} \varphi(b,a) \wedge \tau(b), \\
		(\forall_{\varphi}\tau)(a) & = \alpha(a) \wedge \bigwedge_{b \in B} [\varphi(b,a) \imply \tau(b)]. \qedhere
	\end{align*}
\end{Prop*}

Finally, we describe the higher-order structure of $ \mathbf{Set}(\mathcal{O}(X)) $.

\begin{Prop} \label{prop:subobj-classifier-of-Set(OX)}
	Put $ \delta(U,V) = ( U \imply V ) \wedge (V \imply U) $ for $ U,V \in \mathcal{O}(X) $. Then $ (\mathcal{O}(X),\delta) $ is an $ \mathcal{O}(X) $-valued set.
	Let $ t \colon (\{\ast\},\top) \to (\mathcal{O}(X),\delta) $ be the morphism defined by $ t(\ast,U) = U $.
	This yields a subobject classifier of $ \mathbf{Set}(\mathcal{O}(X)) $.
\end{Prop}
\begin{proof}
	Let $ \chi \colon (A,\alpha) \to (\mathcal{O}(X),\delta) $ be a morphism.
	Since $ t $ corresponds to the strict relation $ \id_{\mathcal{O}(X)} $ on $ (\mathcal{O}(X),\delta) $,
	the pullback of $ t $ along $ \chi $ is given by the strict relation $ \sigma(a) = \bigvee_{U} \chi(a,U) \wedge U $.
	Conversely, given a strict relation $ \sigma $ on $ (A,\alpha) $, then $ \sigma $ itself represents a morphism $ \chi(a,U) = \alpha(a) \wedge (U \leftrightarrow \sigma(a)) $.
	These correspondences yield a bijection between $ \mathcal{P}(A,\alpha) $ and $ \Hom((A,\alpha), (\mathcal{O}(X),\delta)) $.
	
	\[ \tikz[auto]{
		\node (UL) at (0,1.5) {$ (A,\alpha_{\sigma}) $};
		\node (DL) at (0,0) {$ (A,\alpha) $};
		\node (UR) at (3,1.5) {$ (\{\ast\},\top) $};
		\node (DR) at (3,0) {$ (\mathcal{O}(X),\delta) $};
		\node at (0.5,1) {$ \lrcorner $};
		\draw[->] (UL) to node {$ ! $} (UR);
		\draw[>->] (UR) to node {$ t $} (DR);
		\draw[>->] (UL) to node[swap] {$ \iota_{\sigma} $} (DL);
		\draw[->] (DL) to node {$ \chi $} (DR);
	} \qedhere \]
\end{proof}

Similarly to $ \mathcal{O}(X) $, $ \mathcal{P}(A,\alpha) $ is not only a frame but also an $ \mathcal{O}(X) $-valued set.

\begin{Prop}
	The power object of $ (A,\alpha) $ is given by $ \mathcal{P}(A,\alpha) $ equipped with the valuation
	\footnote{Note that $ \bar{\alpha} $ does not necessarily coincide with 
		$ \tilde{\alpha} $ on $ \tilde{A} $ in \textbf{Proposition \ref{prop:completion-of-OX-sets}}.
		Indeed, while $ \tilde{\alpha}(\sigma,\tau) \leq \bar{\alpha}(\sigma,\tau) $ holds for any $ \sigma,\tau \in \tilde{A}$,
		the converse inequality does not hold if $ \sigma=\tau=\sigma_a $ for a non-global element $ a \in A $.
	}
	\[ \bar{\alpha}(\sigma,\tau) := \bigwedge_{a \in A} \sigma(a) \leftrightarrow \tau(a). \]
\end{Prop}
\begin{proof}
	We would like to establish the following bijection
	\[ \Hom((B,\beta),\mathcal{P}(A,\alpha)) \simeq \mathcal{P}((B,\beta)\times(A,\alpha)). \]
	For a morphism $ \varphi \colon (B,\beta) \to \mathcal{P}(A,\alpha) $,
	we have a strict relation $ \theta $ on $ (B,\beta)\times(A,\alpha) $ such that
	\[ \theta(b,a) = \bigvee_{\tau \in \mathcal{P}(A,\alpha)} \varphi(b,\tau) \wedge \tau(a). \]
	On the other hand, for any strict relation $ \theta $, we have a morphism
	\[ \varphi(b,\tau) = \beta(b) \wedge \bigwedge_{a \in A} \theta(b,a) \leftrightarrow \tau(a). \]
	These correspondences are mutual inverses.
\end{proof}

\section{Sheaves of Structures and Heyting-Valued Structures}\label{sec:sheaf-of-str-and-H-valued-str}

\subsection{Structures in a Topos} \label{subsec:structures-in-topos}
We will be concerned with categorical semantics in the toposes $ \mathbf{Sh}(X) $ and $ \mathbf{Set}(\mathcal{O}(X)) $.
In this subsection, we take a glance at first-order categorical logic, which originates from \cite{MakReyFOCL}.
The main reference is \cite{Elephant}, in particular, Chapter D1 in volume 2.
For an overview, we refer the reader to Caramello's account \cite{CarBackground},
which is a preliminary version of the first two chapters of her book \cite{CarTST}.
Although many fragments of (possibly infinitary) first-order logic are considered in the context of categorical logic, we restrict our attention to single-sorted intuitionistic logic.
Examples of other fragments include Horn, cartesian, regular, coherent, classical, and geometric logics.

\begin{Def}
	A \emph{(first-order) language} $ \mathcal{L} $ consists of the following data:
	\begin{itemize}
		\item A set $ \mathcal{L}\mhyphen\mathrm{Func} $ of \emph{function symbols}.
		Each function symbol $ f \in \mathcal{L}\mhyphen\mathrm{Func}  $ is associated with a natural number $ n $ (the \emph{arity} of $ f $).
		If $ n=0 $, $ f $ is called a \emph{constant (symbol)}.
		
		\item A set $ \mathcal{L}\mhyphen\mathrm{Rel} $ of \emph{relation symbols}.
		Each relation symbol $ R \in \mathcal{L}\mhyphen\mathrm{Rel}  $ is associated with a natural number $ n $ (the \emph{arity} of $ R $).
		If $ n=0 $, $ R $ is called an \emph{atomic proposition}. \qedhere
	\end{itemize}
\end{Def}

\emph{$ \mathcal{L} $-terms} and \emph{$ \mathcal{L} $-formulas} are defined as usual. We need some conventions.

\begin{Def}
	\ 
	
	\begin{enumerate}
		\item A \emph{context} is a finite list $ \bvec{u} \equiv u_1,\dots,u_n $ of distinct variables.
		If $ n=0 $, it is called the \emph{empty context} and denoted by $ [\,] $.
		
		\item We say that a context $ \bvec{u} $ is \emph{suitable} for an $ \mathcal{L} $-formula $ \varphi $
		when $ \bvec{u} $ contains all the free variables of $ \varphi $.
		A formula $ \varphi $ equipped with a suitable context $ \bvec{u} $
		is called a \emph{formula-in-context} and indicated by $ \varphi(\bvec{u}) $.
		Similarly, \emph{terms-in-context} can be defined. \qedhere
	\end{enumerate}
\end{Def}

For an $ \mathcal{L} $-formula-in-context $ \varphi(\bvec{u},\bvec{v}) $,
we abbreviate, e.g., $ \exists v_1 \cdots \exists v_n \varphi(\bvec{u},\bvec{v}) $
as $ \exists \bvec{v} \varphi(\bvec{u},\bvec{v}) $.
We also abbreviate the $ \mathcal{L} $-formula $ \bigwedge_i u_i = v_i $ as $ \bvec{u}=\bvec{v} $,
where $ \bvec{u},\bvec{v} $ are assumed to have the same length.
A formula is \emph{closed} if it contains no free variables.

We now give categorical semantics in an arbitrary elementary topos $ \mathcal{E} $,
though we will only need the case when $ \mathcal{E} $ is a Grothendieck topos.

\begin{Def}
	Let $\mathcal{L}$ be a language and $\mathcal{E}$ a topos.
	An $ \mathcal{L} $-structure $ \mathcal{M} $ in $ \mathcal{E} $ is given by specifying the following data:
	\begin{itemize}
		\item We have the underlying object $ \dvert{\mathcal{M}} \in \mathcal{E} $
		and denote the $ n $-ary product by $ \dvert{\mathcal{M}}^{n}  $.
		In particular, $ \dvert{\mathcal{M}}^{0}  $ is the terminal object $ 1_{\mathcal{E}}$.
		
		\item To an $ n $-ary function symbol $ f $,
		we assign a morphism $ f^{\mathcal{M}} \colon \dvert{\mathcal{M}}^{n} \to \dvert{\mathcal{M}}$.
		
		\item To an $ n $-ary relation symbol $R $,
		we assign a subobject $R^{\mathcal{M}} $ of $ \dvert{\mathcal{M}}^n $.
	\end{itemize}
	As usual, we will not distinguish $ \mathcal{M} $ and its underlying object $ \dvert{\mathcal{M}} $ in notation.
\end{Def}

Interpretations of $ \mathcal{L} $-terms and $ \mathcal{L} $-formulas
are defined by using internal operations in $ \mathcal{E} $.

\begin{Def}[Interpretations of terms]
	Let $ \mathcal{M} $ be an $ \mathcal{L} $-structure in a topos $ \mathcal{E} $.
	For an $\mathcal{L}$-term-in-context $ t(\bvec{u}) $,
	we define the interpretation $ t^{\mathcal{M}} \colon \mathcal{M}^n \to \mathcal{M} $ inductively.
	\begin{itemize}
		\item If $ t $ is a variable $ u_i $,
		then $ t^{\mathcal{M}} $ is the $ i $-th product projection $ \pi_i \colon \mathcal{M}^n \to \mathcal{M} $.
		
		\item If interpretations of $ \mathcal{L} $-terms $ t_i(\bvec{u}) $ and $ s(\bvec{v}) $ are given,
		the term $ s(t_1(\bvec{u}),\ldots,t_m(\bvec{u})) $ is interpreted as the composite of the following morphisms:
		\[ \mathcal{M}^n \xrightarrow{\langle t_1^{\mathcal{M}} ,\ldots , t_m^{\mathcal{M}} \rangle}
		\mathcal{M}^m \xrightarrow{s^{\mathcal{M}}} \mathcal{M}, \]
		where $ \langle t_1^{\mathcal{M}} ,\ldots , t_m^{\mathcal{M}}  \rangle $ is the morphism
		obtained from the morphisms $ t_i^{\mathcal{M}} $ by using the universal property of the product $ \mathcal{M}^m $. \qedhere
	\end{itemize}
\end{Def}

\begin{Def}[Interpretations of formulas]
	Let $ \mathcal{M} $ be an $ \mathcal{L} $-structure in a topos $ \mathcal{E} $.
	For an $\mathcal{L}$-formula-in-context $ \varphi(\bvec{u}) $,
	we define the interpretation $ \dbracket{\bvec{u}.\,\varphi}_{\mathcal{M}} $ as a subobject of $ \mathcal{M}^n $ inductively.
	(We drop the subscript $ \mathcal{M} $ if no confusion arises.)
	\begin{itemize}
		\item If $ \varphi \equiv (s(\bvec{u})=t(\bvec{u}))$ where $ s,t $ are terms,
		then $ \dbracket{\bvec{u}.\,\varphi} $ is defined to be the equalizer of
		\[ \tikz[auto]{
			\node (L) at (0,0) {$ \mathcal{M}^n $};
			\node (R) at (4,0) {$ \mathcal{M} $};
			\draw[transform canvas={yshift=4pt},->] (L) to node {$s^{\mathcal{M}}$} (R);
			\draw[transform canvas={yshift=-4pt},->] (L) to node[swap] {$t^{\mathcal{M}}$} (R);
		}.
		\]
		
		\item If $\varphi \equiv R(t_1(\bvec{u}), \ldots ,t_m(\bvec{u}))$, then $ \dbracket{\bvec{u}.\,\varphi} $ is the pullback
		\[ \tikz[auto]{
			\node (UL) at (0,1.5) {$ \dbracket{\bvec{u}.\,R(t_1, \ldots ,t_m)} $};
			\node (DL) at (0,0) {$ \mathcal{M}^n  $};
			\node (UR) at (5,1.5) {$ R^{\mathcal{M}} $};
			\node (DR) at (5,0) {$ \mathcal{M}^m $};
			\draw[>->] (UL) to (DL);
			\draw[->] (UL) to (UR);
			\draw[->] (DL) to node {$ \langle t_1^{\mathcal{M}},\ldots,t_m^{\mathcal{M}}\rangle $} (DR);
			\draw[>->] (UR) to (DR);
		}. \]
		
		\item If $ \varphi \equiv \top , \bot, \psi \land \theta $, $ \psi \lor \theta $, $ \psi \limply \theta $ or $ \lnot \psi $,
		then $ \dbracket{\bvec{u}.\,\varphi} $ is defined as expected by using the Heyting operations on $\Sub(\mathcal{M}^n)$.
		
		\item If $\varphi \equiv \exists v \psi(\bvec{u},v)$,
		then $ \dbracket{\bvec{u}.\,\varphi} $ is the image as in the following diagram:
		\[ \tikz[auto]{
			\node (UL) at (0,1.5) {$ \dbracket{\bvec{u},v.\,\psi} $};
			\node (DL) at (1,0) {$ \dbracket{\bvec{u}.\, \exists v \psi } $};
			\node (UR) at (3,1.5) {$ \mathcal{M}^n \times \mathcal{M}  $};
			\node (DR) at (4,0) {$ \mathcal{M}^n $};
			\draw[->>] (UL) to (DL);
			\draw[>->] (UL) to (UR);
			\draw[>->] (DL) to (DR);
			\draw[->] (UR) to node {$ \pi $} (DR);
		}, \]
		where $ \pi $ is the projection onto $ \mathcal{M}^n $.
		
		\item If $\varphi \equiv \forall v \psi(\bvec{u},v)$,
		then $ \dbracket{\bvec{u}.\,\varphi} := \forall_{\pi} \dbracket{\bvec{u},v.\,\psi} $,
		where $ \forall_{\pi} \colon \Sub(\mathcal{M}^n \times \mathcal{M}) \to \Sub(\mathcal{M}^n) $ is the right adjoint of $ \pi^* $. \qedhere
	\end{itemize}
\end{Def}

In this paper, we will not consider the notions of models of a theory in a topos nor homomorphisms between structures.

\subsection{Sheaves of Structures and Heyting-Valued Structures} \label{subsec:sheaf-of-str-and-H-valued-str}

We now investigate the relationship between structures in $ \mathbf{Sh}(X) $ and those in $ \mathbf{Set}(\mathcal{O}(X)) $.
We first consider the case of sheaves on a \textit{topological space} $ X $.
Let $ \mathbf{LH} $ be the category of topological spaces and local homeomorphisms between them.
Recall that the slice category $ \mathbf{LH}/X $ is categorically equivalent to $ \mathbf{Sh}(X) $.
Comer \cite{Comer1974}, Ellerman \cite{Ell1974}, and Macintyre \cite{Mac1973} used the following notion to obtain model-theoretic results:

\begin{Def} \label{def:sheaf-of-str-on-space}
	A \emph{sheaf of $ \mathcal{L} $-structures (on $ X $)} is a tuple 
	\[ \left (X,E,\pi,\Set{f^{E_x} \smid x \in X,\, f \in \mathcal{L}\mhyphen\mathrm{Func}},
	\Set{R^{E_x} \smid x \in X,\, R \in \mathcal{L}\mhyphen\mathrm{Rel}} \right ) \]
	such that
	\begin{itemize}
		\item $ \pi \colon E \to X $ is a local homeomorphism of topological spaces,
		\item each stalk $ E_x $ equipped with $ \{f^{E_x}\}_f $ and $ \{R^{E_x}\}_R $ is an $ \mathcal{L} $-structure, and
		\begin{itemize}
			\item for each function symbol $ f $, the map $ \coprod_x (E_x)^n \to E $ induced by $ \{f^{E_x}\}_x $ is continuous,
			\item for each relation symbol $ R $, the subset $ \coprod_x R^{E_x} \subseteq \coprod_x (E_x)^n $ is open,
		\end{itemize}
		where $ \coprod_x (E_x)^n $ is seen as a subspace of the product space $ E^n $ for $ n>0 $, and $ \coprod_x (E_x)^0 \simeq X $. \qedhere
	\end{itemize}
\end{Def}

Sheaves of abelian groups or of rings in geometry are, of course, such examples for suitable languages.
We will meet other examples which give model-theoretic constructions of (usual $ \mathbf{Set} $-valued) structures in the next section.

\begin{Lem}
	A sheaf of $ \mathcal{L} $-structures is identified with an $ \mathcal{L} $-structure in $ \mathbf{LH}/X $.
\end{Lem}
\begin{proof}
	Notice the following facts:
	\begin{itemize}
		\item $ \coprod_x (E_x)^n $ is a fiber product $ E \times_X \dots \times_X E $, i.e., a product in $ \mathbf{LH}/X $.
		\item Any monomorphism in $ \mathbf{LH}/X $ is an open embedding. \qedhere
	\end{itemize}
\end{proof}

Hereafter, we fix a locale $ X $.
We will also say ``sheaves of structures on $ X $'' to mean structures in $ \mathbf{Sh}(X) $.
When we mention a subsheaf $ Q $ of a sheaf $ P $, each $ Q(U) $ is assumed to be a subset of $ P(U) $.

Before we define Heyting-valued structures, let us introduce space-saving notations.
If $ (\mathcal{M},\delta) $ is an $ \mathcal{O}(X) $-valued set,
then the $ n $-th power $ \mathcal{M}^n $ is canonically equipped with the valuation as in \textbf{Proposition \ref{prop:fin-limits-in-Set(OX)}}(2).
For tuples $ \bvec{a},\bvec{a'} \in \mathcal{M}^n $, we simply write $ \delta(\bvec{a},\bvec{a'}) $ (resp.\ $ \delta(\bvec{a}) $)
for $ \bigwedge_i \delta(a_i,a'_i) $ (resp.\ $ \bigwedge_i \delta(a_i) $).
These notations are useful, but, in the case $ n=2 $, we will always write $ \delta(a) \wedge \delta(b) $ for $ \delta((a,b),(a,b)) $ 
to avoid confusion between $ \delta(a,b) $ and $ \delta((a,b)) $. 

\begin{Def}
	An \emph{$ \mathcal{O}(X) $-valued $ \mathcal{L} $-structure} is an $ \mathcal{L} $-structure in the topos $ \mathbf{Set}(\mathcal{O}(X)) $,
	i.e., it consists of the following data:
	\begin{itemize}
		\item an $ \mathcal{O}(X) $-valued set $ (\mathcal{M},\delta) $,
		\item for each function symbol $ f $, a morphism $ f^{\mathcal{M}} \colon (\mathcal{M}^n,\delta) \to (\mathcal{M},\delta) $,
		\item for each relation symbol $ R $, a strict relation $ R^{\mathcal{M}} \colon \mathcal{M}^n \to \mathcal{O}(X)$.
	\end{itemize}
	The interpretation of equality is the diagonal $ (\mathcal{M},\delta) \rightarrowtail (\mathcal{M}^2,\delta) $,
	which corresponds to the strict relation $ (a,b) \mapsto \delta(a,b) $ on $ (\mathcal{M}^2,\delta) $ under the bijection in 
	\textbf{Proposition \ref{prop:strict-relations-vs-subobjects}}.
\end{Def}

Fourman \& Scott \cite[p.~365]{FourSco1979} defined Heyting-valued structures in a slightly less general form.
Structures in the topos associated with a tripos are discussed in \cite[p.~69]{vOosten2008}.

Recall the construction of $ \Theta \colon \mathbf{Set}^{\mathcal{O}(X)^{\op}} \to \mathbf{Set}(\mathcal{O}(X)) $ at the beginning of \S\ref{subsec:sheaves-vs-OX-sets}.
We can obtain $ \mathcal{O}(X) $-valued structures from sheaves of structures on $ X $ by applying $ \Theta $.

\begin{Lem}
	Let $ P $ be a \textit{presheaf} on $ X $.
	Then, the $ n $-ary product $ \Theta(P)^n $ is isomorphic to $ \Theta(P^n) $ as $ \mathcal{O}(X) $-valued sets.
	Indeed, the canonical map $ h \colon \coprod_U (PU)^n \to (\coprod_U PU)^n $ represents an isomorphism $ \iota \colon \Theta(P^n) \isoarrow \Theta(P)^n $
	so that $ \iota(\bvec{b},\bvec{a}) = \iota^{-1}(\bvec{a},\bvec{b}) = \delta_P(h(\bvec{b}),\bvec{a}) $.
	
	Moreover, for a strict relation $ \sigma $ on $ \Theta(P^n) $, the corresponding strict relation $ \tau $ on $ \Theta(P)^n $ is given by
	\[ \tau(\bvec{a}) = \bigvee_{\bvec{b} \in \Theta(P^n)} \sigma(\bvec{b}) \wedge \delta_{P}(h(\bvec{b}),\bvec{a}). \]
\end{Lem}
\begin{proof}
	For the case when $ P $ is a sheaf, this lemma is an immediate consequence of the fact that 
	$ \Theta \colon \mathbf{Sh}(X) \to \mathbf{Set}(\mathcal{O}(X)) $ is part of an equivalence of categories.
	We can also see directly that $ \Theta \colon \mathbf{Set}^{\mathcal{O}(X)^{\op}} \to \mathbf{Set}(\mathcal{O}(X)) $ preserves finite products
	by using \textbf{Corollary \ref{cor:isomorphisms-induced-by-maps}}.
	
	For a given $ \sigma $, by the proof of \textbf{Proposition \ref{prop:strict-relations-vs-subobjects}},
	the corresponding subobject of $ \Theta(P^n) $ is $ (\coprod_U (PU)^n,(\delta_{P^n})_{\sigma}) $ 
	with $ (\delta_{P^n})_{\sigma}(\bvec{b},\bvec{b'}) = \sigma(\bvec{b}) \wedge \delta_{P^n}(\bvec{b},\bvec{b'})$.
	Hence, $ \tau $ is given by
	\[ \tau(\bvec{a}) 
	= \bigvee_{\bvec{b},\bvec{b'}\in\Theta(P^n)} (\delta_{P^n})_{\sigma}(\bvec{b},\bvec{b'}) \wedge \iota(\bvec{b'},\bvec{a})
	= \bigvee_{\bvec{b} \in \Theta(P^n)} \sigma(\bvec{b}) \wedge \delta_{P}(h(\bvec{b}),\bvec{a}). \]
	We remark that $ \tau(h(\bvec{b})) = \sigma(\bvec{b}) $ for any $ \bvec{b} \in \Theta(P^n) $ and therefore $ \tau $ is an extension of $ \sigma $ along $ h $.
\end{proof}

\begin{Prop} \label{prop:Heyt-str-obtained-from-sheaf-of-str}
	If $ P $ is a sheaf of $ \mathcal{L} $-structures on $ X $, then we can make the $ \mathcal{O}(X) $-valued set $ \Theta(P) $
	into an $ \mathcal{O}(X) $-valued $ \mathcal{L} $-structure canonically.
\end{Prop}
\begin{proof}
	Here we describe in detail the corresponding $ \mathcal{O}(X) $-valued $ \mathcal{L} $-structure $ \mathcal{M} $.
	For each function $ f $, we have a morphism $ f^P \colon P^n \to P $ of sheaves. This induces a morphism $ \Theta(f^P) \colon \Theta(P^n) \to \Theta(P) $.
	By the previous lemma, we obtain a morphism $ f^{\mathcal{M}} \colon \Theta(P)^n \to \Theta(P) $, which can be computed as
	\[ f^{\mathcal{M}}(\bvec{a},a') = \bigvee_{\bvec{b} \in \Theta(P^n)} \iota^{-1}(\bvec{a},\bvec{b}) \wedge \Theta(f^P)(\bvec{b},a') 
	= \delta_P(f^P(\bvec{a}|_{\delta_P(\bvec{a})}),a')\]
	for $ \bvec{a} \in \Theta(P)^n $ and $ a' \in \Theta(P) $,
	where $ \bvec{a}|_{\delta_P(\bvec{a})} = (a_1|_{\delta_P(\bvec{a})},\dots,a_n|_{\delta_P(\bvec{a})}) $.
	In particular, $ f^{\mathcal{M}} $ is represented by the map $ k \colon \mathcal{M}^n \to \mathcal{M} $ with $ k(\bvec{a}) = f^P(\bvec{a}|_{\delta_P(\bvec{a})}) $.
	
	For each relation $ R $, we have a subsheaf $ R^P \rightarrowtail P^n $. This induces a subobject $ \Theta(R^P) \rightarrowtail \Theta(P^n) $,
	which corresponds to the following strict relation $ \sigma \colon \coprod_U (PU)^n \to \mathcal{O}(X) $: for $ \bvec{b} \in (PU)^n $,
	\[ \sigma(\bvec{b}) = \bigvee_{\bvec{b'}\in \Theta(R^P)} \delta_{P^n}(\bvec{b},\bvec{b'})=\bigvee \Set{W \leq U \smid \bvec{b}|_W \in R^P(W)}.\]
	By the previous lemma, we obtain a subobject of $ \Theta(P)^n $,
	which corresponds to the following strict relation $ R^{\mathcal{M}} \colon \mathcal{M}^n \to \mathcal{O}(X) $: 
	for $ \bvec{a} \in PU_1 \times \dots \times PU_n $,
	\begin{align*}
		R^{\mathcal{M}}(\bvec{a}) 
		& = \bigvee_{\bvec{b} \in \Theta(P^n)} \iota^{-1}(\bvec{a},\bvec{b}) \wedge \sigma(\bvec{b}) \\
		& = \bigvee_{\bvec{b} \in \Theta(P^n)} \left [ \delta_{P}(h(\bvec{b}),\bvec{a}) \wedge \bigvee \Set{W \leq \delta_{P^n}(\bvec{b}) \smid \bvec{b}|_W \in R^P(W)} \right ]\\
		& = \bigvee  \Set{W \leq U_1 \wedge \dots \wedge U_n \smid \bvec{a}|_W \in R^P(W)}.
	\end{align*}
	
	Notice that the subobject $ \Theta(P) \rightarrowtail \Theta(P)^2 $ obtained from the diagonal $ P \rightarrowtail P^2 $ is 
	the same as the one determined by the strict relation $ \delta_P $ on $ \Theta(P) $.
\end{proof}

We could describe the converse construction (from Heyting-valued structures to sheaves of structures).
This involves a complicated use of completion of Heyting-valued sets, and we do not find such details to be useful for the purpose of this paper.
So we skip it at this point.

In the context of set theory (e.g.\ \cite{PieVia2020}), there are examples of Heyting-valued structures which do not come from sheaves.

\begin{Rmk}
	Some authors have applied (set-theoretic) Boolean-valued universes to mathematics (cf.\ \cite{KusKutaBVA} and the references therein).
	From the viewpoint of \textbf{Remark \ref{rmk:relation-to-set-theory}}, these works complement our understanding of Heyting-valued structures.
\end{Rmk}

\subsection{Forcing Values of Formulas}  \label{subsec:forcing-val}
Forcing values of formulas derive from Boolean-valued set theory.
Here we first define them categorically and then observe that our definition is compatible with the usual one.
The categorical description seems to be folklore but has not appeared in an explicit form elsewhere.
For an $ \mathcal{O}(X) $-valued $ \mathcal{L} $-structure $ (\mathcal{M},\delta) $,
we write $ \mathcal{L}_{\mathcal{M}} $ for the language extending $ \mathcal{L} $ by adding a new constant symbol for each element of $ \mathcal{M} $.

\begin{Def}
	For an $ \mathcal{L} $-formula-in-context $ \varphi(\bvec{u}) $,
	the strict relation $ \dVert{\varphi(-)}^{\mathcal{M}} $ on $ (\mathcal{M},\delta)^n $ is defined to be 
	the one corresponding to the subobject $ \dbracket{\bvec{u}.\,\varphi}_{(\mathcal{M},\delta)} \rightarrowtail (\mathcal{M},\delta)^n $.
	For $ \bvec{a} \in \mathcal{M}^n $, $ \dVert{\varphi(\bvec{a})}^{\mathcal{M}} $ is called 
	the \emph{forcing value} of the closed $ \mathcal{L}_{\mathcal{M}} $-formula $ \varphi(\bvec{a}) $.
	We drop the superscript $ \mathcal{M} $ if no confusion arises.
\end{Def}

Since the strict relation $ \bvec{a} \mapsto \delta(\bvec{a}) $ is the greatest element in $ \mathcal{P}(\mathcal{M}^n,\delta) $,
$ \dVert{\varphi(\bvec{a})}^{\mathcal{M}} \leq \delta(\bvec{a}) $ always holds.
Using the results in \S\ref{subsec:topos-structure-of-Set(OX)}, we can calculate the forcing values inductively.

\begin{Prop*} \label{prop:inductive-calc-of-forcing-values}
	\begin{align*}
		\dVert{R(t_1(\bvec{a}), \ldots ,t_m(\bvec{a}))}^{\mathcal{M}} 
		& = \bigvee_{\bvec{b} \in \mathcal{M}^m} \langle t_1^{\mathcal{M}}, \dots, t_m^{\mathcal{M}} \rangle(\bvec{a},\bvec{b}) \wedge R^{\mathcal{M}}(\bvec{b}), \\
		\dVert{s(\bvec{a})=t(\bvec{a})}^{\mathcal{M}} 
		& = \bigvee_{b,c \in \mathcal{M}} \langle s^{\mathcal{M}},t^{\mathcal{M}} \rangle(\bvec{a},(b,c)) \wedge \delta(b,c), \\
		\dVert{\varphi(\bvec{a}) \land \psi(\bvec{a})}^{\mathcal{M}} 
		& = \dVert{\varphi(\bvec{a})}^{\mathcal{M}} \wedge  \dVert{\psi(\bvec{a})}^{\mathcal{M}}, \\
		\dVert{\varphi(\bvec{a}) \lor \psi(\bvec{a})}^{\mathcal{M}} 
		& = \dVert{\varphi(\bvec{a})}^{\mathcal{M}} \vee \dVert{\psi(\bvec{a})}^{\mathcal{M}}, \\
		\dVert{\varphi(\bvec{a}) \limply \psi(\bvec{a})}^{\mathcal{M}} 
		& = \delta(\bvec{a}) \wedge \left [ \dVert{\varphi(\bvec{a})}^{\mathcal{M}} \imply \dVert{\psi(\bvec{a})}^{\mathcal{M}} \right ], \\
		\dVert{\exists v \varphi(\bvec{a},v)}^{\mathcal{M}} 
		& = \bigvee_{b \in \mathcal{M}} \dVert{\varphi(\bvec{a},b)}^{\mathcal{M}} , \\
		\dVert{\forall v \varphi(\bvec{a},v)}^{\mathcal{M}} 
		& = \delta(\bvec{a}) \wedge \bigwedge_{b \in \mathcal{M}} \left [ \delta(b) \imply \dVert{\varphi(\bvec{a},b)}^{\mathcal{M}} \right ]. \qedhere
	\end{align*}
\end{Prop*}
\begin{Rmk}
	If a formula $ \varphi $ has a suitable context $ \bvec{u} $ and $ v $ is a variable distinct from $ \bvec{u} $,
	we have to distinguish the formulas-in-context $ \varphi(\bvec{u}) $ and $ \varphi(\bvec{u},v) $.
	Indeed, the forcing values $ \dVert{\varphi(\bvec{a})} $ and $ \dVert{\varphi(\bvec{a},b)} $ can be different and
	\[ \dVert{\varphi(\bvec{a},b)} = \dVert{\varphi(\bvec{a})} \wedge \delta(b). \qedhere \]
\end{Rmk}
This description of forcing values is compatible with those in \cite[Definition 5.13]{FourSco1979}, \cite[Definition 13.6.6]{TvD} and \cite[p.~70]{vOosten2008}.
The soundness and completeness theorems for Heyting-valued semantics are usually formulated 
with respect to intuitionistic predicate logic \textit{with existence predicate} (for short, \textbf{IQCE}) as in \cite[\S2.2, \S13.6]{TvD}.
However, we will only need soundness of the following form:

\begin{Lem} \label{lem:soundness-for-forcing-values}
	If the sentence $ \forall \bvec{u} [\varphi(\bvec{u}) \limply \psi(\bvec{u})] $ is intuitionistically valid,
	then $ \dVert{\varphi(\bvec{a})}^{\mathcal{M}} \leq \dVert{\psi(\bvec{a})}^{\mathcal{M}} \leq \delta(\bvec{a})$ holds for any $ \bvec{a} \in \mathcal{M}^n $.
\end{Lem}
\begin{proof}
	The assumption implies $ \dbracket{\bvec{u}.\,\varphi} \leq \dbracket{\bvec{u}.\,\psi} $ as subobjects of $ (\mathcal{M},\delta)^n $.
	Therefore, the conclusion holds by the definition of forcing values.
\end{proof}

Let $ P $ be a sheaf of $ \mathcal{L} $-structures and $ \Theta(P)=(\mathcal{M},\delta) $ 
be the $ \mathcal{O}(X) $-valued $ \mathcal{L} $-structure obtained from \textbf{Proposition \ref{prop:Heyt-str-obtained-from-sheaf-of-str}}.
We can see
\begin{enumerate}
	\item For any $ \mathcal{L} $-term $ t(\bvec{u}) $, the morphism $ t^{\mathcal{M}} $ is represented by 
	the map $ \mathcal{M}^n \ni \bvec{a} \mapsto t^P(\bvec{a}|_{\delta(\bvec{a})}) \in \mathcal{M} $
	where $ t^P \colon P^n \to P $ is the interpretation of $ t $ by $ P $.
	
	\item For any atomic $ \mathcal{L} $-formula $ R(t_1(\bvec{u}),\dots,t_m(\bvec{u})) $ and $ \bvec{a} \in \mathcal{M}^n $,
	\begin{align*}
		& \hphantom{=}\; \dVert{R(t_1(\bvec{a}), \ldots ,t_m(\bvec{a}))}^{\mathcal{M}} \\
		& = R^{\mathcal{M}}( t_1^{P}(\bvec{a}|_{\delta(\bvec{a})}),\dots,t_m^{P}(\bvec{a}|_{\delta(\bvec{a})}) ) \\
		& = \bigvee \Set{
			W \leq \delta(\bvec{a}) \smid ( t_1^{P}(\bvec{a}|_{W}),\dots,t_m^{P}(\bvec{a}|_{W}) ) \in R^P(W)
		}.
	\end{align*}
	Similarly for the formula $ s(\bvec{u}) = t(\bvec{u}) $.
\end{enumerate}

More generally, the forcing value $ \dVert{\varphi(-)} $ for $ \Theta(P) $ can be described in terms of the subsheaf $ \dbracket{\bvec{u}.\,\varphi} $ of $ P^n $.
Let $ \Omega $ be the sheaf $ U \mapsto \Omega(U)=(U){\downarrow} $.
This is a subobject classifier in $ \mathbf{Sh}(X) $, and
we thus obtain the characteristic morphism $ \chi \colon P^n \to \Omega $ by the universality of the subobject classifier:
\[ \tikz[auto,baseline=0.75cm]{
	\node (UL) at (0,1.5) {$ \dbracket{\bvec{u}.\,\varphi} $};
	\node at (0.4,1.1) {$ \lrcorner $};
	\node (UR) at (1.5,1.5) {$ 1 $};
	\node (DL) at (0,0) {$ P^n $};
	\node (DR) at (1.5,0) {$ \Omega $};
	\draw[>->] (UL) to (DL);
	\draw[->] (UL) to node {$ ! $} (UR);
	\draw[->] (DL) to node {$ \chi $} (DR);
	\draw[->] (UR) to node {$ \mathrm{true} $} (DR);
} \qquad 
\chi_U(\bvec{a}) = \bigvee\Set{W \leq U \smid \bvec{a}|_W \in \dbracket{\bvec{u}.\,\varphi}(W)}. 
\]

Using \textbf{Proposition \ref{prop:subobj-classifier-of-Set(OX)}} and 
the fact that $ \Theta(\Omega) $ and $ (\mathcal{O}(X),\delta) $ in that proposition are canonically isomorphic,
we can verify the following:

\begin{Prop*}[definable subsheaves and forcing values]
	In the above notation, $ \chi_U(\bvec{a}) = \dVert{\varphi(\bvec{a})}^{\mathcal{M}} $ for any $ \bvec{a} \in P^n U $.
	We will denote $ \chi $ by $ \dVert{\varphi(-)}^{P} $ and its component $ \chi_U $ by $ \dVert{\varphi(-)}^{P}_U $.
\end{Prop*}

Let $ \mathbf{y} \colon \mathcal{O}(X) \to \mathbf{Set}^{\mathcal{O}(X)^{\op}} $ be the Yoneda embedding,
and $ \mathbf{a} \colon \mathbf{Set}^{\mathcal{O}(X)^{\op}} \to \mathbf{Sh}(X) $ the associated sheaf functor.
We write $ \bvec{a} \colon \mathbf{ay}U \to P^n $ for the morphism corresponding to $ \bvec{a} \in P^n U $ under the bijection
\[ P^n U \simeq \Hom_{\mathbf{Set}^{\mathcal{O}(X)^{\op}}}(\mathbf{y}U,P^n) \simeq \Hom_{\mathbf{Sh}(X)}(\mathbf{ay}U,P^n). \]
In terms of forcing values, the sheaf semantics in $ \mathbf{Sh}(X) $ (cf.~\cite[\S VI.7]{SGL}) has a simple description:
\begin{align*}
	& \hphantom{\iff}\;\; U \Vdash_P \varphi(\bvec{a}) \\
	& \defarrow \text{the morphism $ \bvec{a} \colon \mathbf{ay}U \to P^n $ factors through the subsheaf $ \dbracket{\bvec{u}.\,\varphi} \rightarrowtail P^n $,} \\
	& \iff \dVert{\varphi(-)}^P \circ \bvec{a} = \mathrm{true} \circ {!}, \\
	& \iff \dVert{\varphi(\bvec{a})}^P_U = U.
\end{align*}
This is the reason why we use the term ``forcing values'' similarly as in \cite{Ell1974}.
Using the above description, we can show the properties of forcing relation \cite[Theorem VI.7.1]{SGL} for the usual site on $ \mathcal{O}(X) $.

\section{Filter-Quotients of Heyting-Valued Structures and {\L}o\'{s}'s Theorem} \label{sec:filt-quot-and-Los-thm}
As we promised after \textbf{Definition \ref{def:sheaf-of-str-on-space}},
we will observe that sheaves of structures give some constructions in model theory.
These constructions can be generalized to constructions for Heyting-valued structures,
and they provide an adequate setup to state our {\L}o\'{s}-type theorem.

\subsection{Model-Theoretic Constructions via Sheaves of Structures}

\begin{Def} \label{def:sections-and-quotients-of-sheaves}
	Let $ P $ be a sheaf of $ \mathcal{L} $-structures on a locale $ X $.
	\begin{enumerate}[label=(\arabic*)]
		\item We make the set $ P(U) $ for a fixed $ U $ into an $ \mathcal{L} $-structure as follows:
		\[ f^{P(U)}(\bvec{a}):=(f^P)_U(\bvec{a}), \;\text{and}\; P(U) \models R(\bvec{a}) \defarrow \bvec{a} \in R^P(U) \subseteq P(U)^m. \]
		
		\item For a filter $ \mathfrak{f} $ on $ \mathcal{O}(X) $, the colimit $ P/\mathfrak{f} :=\rlim_{U \in \mathfrak{f}} P(U) $
		is the quotient of $ \coprod_{U \in \mathfrak{f}} P(U) $ by the following equivalence relation: for $ U,V \in \mathfrak{f} $ and $ a \in P(U) , b \in P(V)$,
		\[ (U,a) \sim (V,b) \defarrow \exists W \in \mathfrak{f},\, W \leq U \wedge V \;\text{and}\; a|_W=b|_W. \]
		We often write $ [\bvec{a}]_{\mathfrak{f}} $ for a tuple $ ([a_1]_{\mathfrak{f}},\dots,[a_n]_{\mathfrak{f}}) $ of equivalence classes.
		Let $ \delta $ be the valuation of $ \Theta(P) $.
		We make $ P/\mathfrak{f} $ into an $ \mathcal{L} $-structure as follows:
		\begin{align*}
			f^{P/\mathfrak{f}}([\bvec{a}]_{\mathfrak{f}}) & := [f^P(a_1|_{\delta(\bvec{a})},\dots,a_n|_{\delta(\bvec{a})})]_{\mathfrak{f}}, \\
			P/\mathfrak{f} \models R([\bvec{a}]_{\mathfrak{f}}) & \defarrow \exists W \in \mathfrak{f},\, \bvec{a}|_W \in R^P(W), \\
			&\iff \exists W \in \mathfrak{f},\, P(W) \models R(\bvec{a}|_W).
		\end{align*}
		
		In particular, if $ X $ is a topological space and $ x \in X $,
		each stalk $ P_x $ is the quotient $ P/\mathfrak{n}_x $ by the filter $ \mathfrak{n}_x $ of open neighborhoods of $ x $. \qedhere
	\end{enumerate}
\end{Def}

\begin{Expl}[Products] \label{expl:products}
	Let $ X $ be a set. 
	Given an $ X $-indexed family $ \{ \mathcal{M}_x \}_{x \in X} $ of $ \mathcal{L} $-structures,
	the product $ \mathcal{N} := \prod_{x \in X} \mathcal{M}_x $ is an $ \mathcal{L} $-structure such that,
	for any elements $ a^i= \{ a^i_x \}_{x \in X} $,
	\begin{gather*}
		f^{\mathcal{N}}(a^1,\dots,a^n) := \left \{ f^{\mathcal{M}_x}(a^1_x,\dots,a^n_x) \right \}_{x \in X}, \\
		\mathcal{N} \models R(a^1,\dots,a^n) \defarrow \forall x \in X,\, \mathcal{M}_x \models R(a^1_x,\dots,a^n_x). 
	\end{gather*}
	Giving an $ X $-indexed family of $ \mathcal{L} $-structures 
	is the same as giving a sheaf of $ \mathcal{L} $-structures on the discrete space $ X $.
	Let $ P $ be the sheaf corresponding to the local homeomorphism $ \coprod_{x \in X} \mathcal{M}_x \to X $ given by the canonical projection.
	Then, the $ \mathcal{L} $-structure $ P(X) $ of global sections is the same as $ \mathcal{N} $.
	
	Notice that, by induction based on \textbf{Proposition \ref{prop:inductive-calc-of-forcing-values}},
	\[ \dVert{\varphi(a^1,\dots,a^n)}^{\Theta(P)} =  \Set{ x \in X \smid \mathcal{M}_x \models \varphi(a^1_x,\dots,a^n_x) } \]
	holds for any formula $ \varphi $ and $ a^1,\dots,a^n \in \mathcal{N} $.
\end{Expl}

\begin{Expl}[Ultraproducts]
	Let $ \mathfrak{u} $ be an ultrafilter over a set $ X $.
	In the same notation as the previous example,
	the ultraproduct $ \prod_{x} \mathcal{M}_x/\mathfrak{u} $ is the quotient of $ \prod_{x} \mathcal{M}_x $ by the equivalence relation
	\[ a \sim b \defarrow \Set{ x \in X \smid a_x = b_x } \in \mathfrak{u} \]
	equipped with canonical interpretations of $ \mathcal{L} $, e.g.,
	\[ \prod_{x} \mathcal{M}_x/\mathfrak{u} \models R([a^1]_{\mathfrak{u}},\dots,[a^n]_{\mathfrak{u}}) \defarrow 
	\Set{ x \in X \smid \mathcal{M}_x \models R(a^1_x,\dots,a^n_x) } \in \mathfrak{u}. \]
	If each $ \mathcal{M}_x $ is non-empty,
	$ \prod_{x} \mathcal{M}_x/\mathfrak{u} $ can be described as a filter-quotient of the sheaf $ P $ corresponding to $ \coprod_{x \in X} \mathcal{M}_x \to X $.
	Since $ P(U) = \prod_{x \in U} \mathcal{M}_x $ and each local section can be extended to a global section by non-emptiness, we have
	\[ \prod_{x} \mathcal{M}_x/\mathfrak{u} \simeq \rlim_{U \in \mathfrak{u}} P(U) = P/\mathfrak{u} .\]
	Thus,  it is reasonable to regard $ P/\mathfrak{u} $ as a ``generalized'' ultraproduct for any $ \mathfrak{u} $ (cf.~\S\ref{subsubsec:Ellerman's-viewpoint}).
	Notice that we need the axiom of choice to extend local sections to global ones,
	but we do not need AC if $ \mathcal{L} $ contains a constant symbol.
\end{Expl}

\begin{Expl}[Bounded Boolean Powers]
	Let $ B $ be a Boolean algebra and $ \mathcal{M} $ be an $ \mathcal{L} $-structure.
	We then have the sheaf $ P $ on the Stone space $ X $ dual to $ B $ determined by
	\[ P(U) := \Set{s \colon U \to \mathcal{M} \smid \text{locally constant map}}. \]
	This becomes a sheaf of $ \mathcal{L} $-structures, and
	$ \mathcal{M}[B]_{\omega}:=P(X) $ is said to be the \emph{bounded Boolean power} of $ \mathcal{M} $ (cf.~\cite[\S9.7]{HodgesMT}).
\end{Expl}

\begin{Expl}[Bounded Boolean Ultrapowers] \label{expl:bdd-boolean-ultrapower}
	In the same notation as the previous example, for any $ s,t \in \mathcal{M}[B]_{\omega}$, the subsets
	\begin{align*}
		\dVert{ R(s_1,\dots,s_n) } & = \Set{ \mathfrak{v} \in X \smid \mathcal{M} \models R(s_1(\mathfrak{v}),\dots,s_n(\mathfrak{v})) }, \\
		\dVert{ s = t } & = \Set{\mathfrak{v} \in X \smid s(\mathfrak{v}) = t(\mathfrak{v}) }
	\end{align*}
	are clopen and identified with elements of $ B $.
	Let $ \mathfrak{u} $ be an ultrafilter on $ B $ ($ = $ a point of $ X $).
	The \emph{bounded Boolean ultrapower} $ \mathcal{M}[B]_{\omega}/\mathfrak{u} $ is given by
	\begin{gather*}
		s \sim t \defarrow \dVert{ s = t } \in \mathfrak{u}, \\
		\mathcal{M}[B]_{\omega}/\mathfrak{u} \models R([s_1]_{\mathfrak{u}},\dots,[s_n]_{\mathfrak{u}}) 
		\defarrow \dVert{ R(s_1,\dots,s_n) } \in \mathfrak{u}.
	\end{gather*}
	$ \mathcal{M}[B]_{\omega}/\mathfrak{u} $ has a representation as a filter-quotient
	\[ \mathcal{M}[B]_{\omega}/\mathfrak{u} \simeq \rlim_{U \in \mathfrak{u}} P(D_U) \simeq P_{\mathfrak{u}}, \]
	where $ D_U = \Set{\mathfrak{v} \in X \smid U \in \mathfrak{v}} $ and $ P_{\mathfrak{u}} $ is the stalk over $ \mathfrak{u} $.
\end{Expl}
Bounded Boolean (ultra)powers are not direct generalizations of ordinary (ultra)powers.
\emph{Unbounded Boolean (ultra)powers} are such things, while they involve more complicated sheaf-theoretic constructions.
Fish \cite{FishMThesis} gives a survey of bounded and unbounded Boolean (ultra)powers.
These constructions can be further generalized to the notion of \emph{Boolean product} (see \cite{BurWer1979}, \cite{Werner1982}, and \cite{BurSankUA}),
which involves sheaves on Stone spaces.

\subsection{Filter-Quotients of Heyting-Valued Structures}
We will generalize the construction of $ P/\mathfrak{f} $ to Heyting-valued structures.
We use filter-quotients of Heyting-valued sets (or structures), which appeared in, e.g., \cite[Definition 2.6]{PieVia2020} and \cite[Chapter 34]{Mira2020}.
Let $ (\mathcal{M},\delta) $ be an $ \mathcal{O}(X) $-valued $ \mathcal{L} $-structure.
Given a filter $ \mathfrak{f} $ on $ \mathcal{O}(X) $,
an $ (\mathcal{O}(X)/\mathfrak{f}) $-valued $ \mathcal{L} $-structure $ \mathcal{M}/\mathfrak{f} $ is defined as follows:
\footnote{We cannot consider a colimit $ \rlim_{U \in \mathfrak{f}} \Set{a \in \mathcal{M} \smid \delta(a) = U} $ as in the case of sheaves since restrictions do not necessarily exist.}
we first observe 
\begin{Claim}
	The following relation $ \sim_{\mathfrak{f}} $ on $ \mathcal{M} $ is an equivalence relation
	\[ a \sim_{\mathfrak{f}} b \defarrow \left [ \delta(a) \vee \delta(b) \imply \delta(a,b) \right ] \in \mathfrak{f}. \]
\end{Claim}
\begin{proof}
	For transitivity, observe
	\begin{align*}
		& \phantom{=} \left ( \delta(a) \vee \delta(b) \imply \delta(a,b) \right ) \wedge \left ( \delta(b) \vee \delta(c) \imply \delta(b,c) \right ) \wedge (\delta(a) \vee \delta(c)) \\
		& = \left [ \delta(a) \wedge \left ( \delta(a) \vee \delta(b) \imply \delta(a,b) \right ) \wedge \left ( \delta(b) \vee \delta(c) \imply \delta(b,c) \right ) \right ] \\
		&\qquad \vee \left [ \delta(c) \wedge \left ( \delta(a) \vee \delta(b) \imply \delta(a,b) \right ) \wedge \left ( \delta(b) \vee \delta(c) \imply \delta(b,c) \right ) \right ] \\
		\intertext{(by using $ \delta(a,b) \leq \delta(a) \wedge \delta(b) $ etc.,)}
		& = \delta(a,b) \wedge \delta(b,c) \leq \delta(a,c).
	\end{align*}
	We then have 
	\[ \left ( \delta(a) \vee \delta(b) \imply \delta(a,b) \right ) \wedge \left ( \delta(b) \vee \delta(c) \imply \delta(b,c) \right ) \leq \delta(a) \vee \delta(c) \imply \delta(a,c). \qedhere \]
\end{proof}
We denote the quotient $ \mathcal{M}/{\sim_{\mathfrak{f}}} $ by $ \mathcal{M}/\mathfrak{f} $ and the equivalence class of $ a \in \mathcal{M} $ by $ [a]_{\mathfrak{f}} $.
In particular, by applying this to the $ \mathcal{O}(X) $-valued set $ (\mathcal{O}(X),\wedge) $,
for which $ U \sim_{\mathfrak{f}} V $ iff $ ( U \leftrightarrow V ) \in \mathfrak{f} $,
we have the quotient Heyting algebra $ \mathcal{O}(X)/\mathfrak{f} $.
By defining the valuation
\footnote{Notice that we use the same notations $ \sim_{\mathfrak{f}} $ and $ [-]_{\mathfrak{f}} $
	for two different equivalence relations on $ \mathcal{M} $ and $ \mathcal{O}(X) $.}
\[ \delta_{\mathfrak{f}}([a]_{\mathfrak{f}},[b]_{\mathfrak{f}}) := [\delta(a,b)]_{\mathfrak{f}}, \]
we can make $ \mathcal{M}/\mathfrak{f} $ into an $ (\mathcal{O}(X)/\mathfrak{f}) $-valued set except that $ \mathcal{O}(X)/\mathfrak{f} $ is not necessarily complete.
We may use the Dedekind--MacNeille completion of $ \mathcal{O}(X)/\mathfrak{f} $ (cf.~\cite[III.3.11]{JohSS})
to define forcing values as in \cite[Definition 2.2]{PieVia2020},
but such a complication will not be necessary for this paper because we will use forcing values 
$ \dVert{\varphi([\bvec{a}]_{\mathfrak{f}})}^{\mathcal{M}/\mathfrak{f}} $ only for atomic formulas $ \varphi $.

For each function $ f $ and each relation $ R $, the morphism 
$ f^{\mathcal{M}/\mathfrak{f}} \colon ((\mathcal{M}/\mathfrak{f})^n,\delta_{\mathfrak{f}}) \to (\mathcal{M}/\mathfrak{f},\delta_{\mathfrak{f}}) $
and the strict relation
$ R^{\mathcal{M}/\mathfrak{f}} \colon (\mathcal{M}/\mathfrak{f})^n \to \mathcal{O}(X)/\mathfrak{f} $
are defined canonically:
\begin{align*}
	f^{\mathcal{M}/\mathfrak{f}}([a_1]_{\mathfrak{f}},\dots,[a_n]_{\mathfrak{f}},[b]_{\mathfrak{f}}) & := [f^{\mathcal{M}}(a_1,\dots,a_n,b)]_{\mathfrak{f}}, \\
	R^{\mathcal{M}/\mathfrak{f}}([a_1]_{\mathfrak{f}},\dots,[a_n]_{\mathfrak{f}}) & := [R^{\mathcal{M}}(a_1,\dots,a_n)]_{\mathfrak{f}}.
\end{align*}
We have finished the construction of the $ (\mathcal{O}(X)/\mathfrak{f}) $-valued $ \mathcal{L} $-structure $ \mathcal{M}/\mathfrak{f} $.
We will call it the \emph{filter-quotient} of $ \mathcal{M} $ by $ \mathfrak{f} $.

Next, we consider filter-quotients of $ \Theta(P) $ for a sheaf $ P $ of $ \mathcal{L} $-structures.
Recall that we already defined an $ \mathcal{L} $-structure $ P/\mathfrak{f} $.

\begin{Lem} \label{lem:P/f-vs-Gamma(M/f)-as-sets}
	Let $ P $ be a sheaf of $ \mathcal{L} $-structures and $ (\mathcal{M},\delta) := \Theta(P) $.
	Then the canonical map $ P/\mathfrak{f} \to \mathcal{M}/\mathfrak{f} $ induces a bijection 
	between $ P/\mathfrak{f} $ and the set of global elements of $ \mathcal{M}/\mathfrak{f} $.
\end{Lem}
\begin{proof}
	Since $ \delta_{\mathfrak{f}}([a]_{\mathfrak{f}}) = [1_X]_{\mathfrak{f}} $ iff $ \delta(a) \in \mathfrak{f} $,
	it is obvious that the image of the canonical map $ P/\mathfrak{f} \to \mathcal{M}/\mathfrak{f} $ consists of global elements.
	We will show this map is injective.
	
	For $ a,b \in \mathcal{M} $ with $ \delta(a),\delta(b) \in \mathfrak{f} $,
	they belong to the same equivalence class in $ P/\mathfrak{f} = \rlim_{U \in \mathfrak{f}} P(U) $ if and only if
	there exists $ U \in \mathfrak{f} $ such that $ U \leq \delta(a) \land \delta(b) $ and $ a|_U = b|_U $.
	On the other hand, by \cite[Proposition 4.7(viii)]{FourSco1979} and separatedness,
	\begin{align*}
		a \sim_{\mathfrak{f}} b & \iff \left [ \delta(a) \vee \delta(b) \imply \bigvee \Set{
			W \leq \delta(a) \wedge \delta(b) \smid a|_W=b|_W
		} \right ] \in \mathfrak{f}, \\
		& \iff W_0 := \bigvee \Set{W \in \mathcal{O}(X) \smid a|_{\delta(a)\wedge W}=b|_{\delta(b)\wedge W}} \in \mathfrak{f}.
	\end{align*}
	Again by separatedness, $ a|_{\delta(a)\wedge W_0}=b|_{\delta(b)\wedge W_0} $.
	Thus, the map $ P/\mathfrak{f} \to \mathcal{M}/\mathfrak{f} $ is injective.
\end{proof}

Note that the map $ P/\mathfrak{f} \to \mathcal{M}/\mathfrak{f} $ is not surjective even if $ X $ is a discrete space.

To give a generalization of the construction of $ P/\mathfrak{f} $ to Heyting-valued structures,
we need to discuss how and when an ordinary structure can be obtained from some ``local sections'' of a Heyting-valued structure.
The following construction is an analogue of \textbf{Definition \ref{def:sections-and-quotients-of-sheaves}}(1).

From an $ \mathcal{O}(X) $-valued $ \mathcal{L} $-structure $ (\mathcal{M},\delta) $,
we would like to construct an ordinary $ \mathcal{L} $-structure $ \Gamma(U,\mathcal{M}) $ as follows.
Set $ \Gamma(U,\mathcal{M}) := \Set{a \in \mathcal{M} \smid \delta(a) =U} $ for $ U \in \mathcal{O}(X)$.
We would like to make $ \Gamma(U,\mathcal{M}) $ into an $\mathcal{L}$-structure so that,
for any relation $ R $ and any $ \bvec{a} \in \Gamma(U,\mathcal{M})^n $,
\[ \Gamma(U,\mathcal{M}) \models R(\bvec{a}) \defarrow R^{\mathcal{M}}(\bvec{a}) = U.\]
To define an interpretation $ f^{\Gamma(U,\mathcal{M})} \colon \Gamma(U,\mathcal{M})^n \to \Gamma(U,\mathcal{M}) $
for each function symbol $ f $, we have to demand the following:
\begin{mdframed}
	\noindent\textbf{Assumption}
	
	For each function symbol $ f $, the morphism $ f^{\mathcal{M}} \colon (\mathcal{M}^n,\delta) \to (\mathcal{M},\delta) $
	is represented by some map $ h \colon \mathcal{M}^n \to \mathcal{M} $ satisfying 
	$ \delta(\bvec{a},\bvec{a'}) \leq \delta(h(\bvec{a}),h(\bvec{a'})) $ and $ \delta(h(\bvec{a})) = \delta(\bvec{a}) $ for any $ \bvec{a},\bvec{a'} $.
\end{mdframed}
By the observation we made in the definition of the functor $ \Theta $ at the beginning of \S\ref{subsec:sheaves-vs-OX-sets},
any Heyting-valued structure of the form $ \Theta(P) $ satisfies the \textbf{Assumption}.
For $ \mathcal{M} $ satisfying the \textbf{Assumption}, we can suitably define $ f^{\Gamma(U,\mathcal{M})} $ 
to be the restriction of $ h $ to $ \Gamma(U,\mathcal{M}) $ and obtain an $ \mathcal{L} $-structure $ \Gamma(U,\mathcal{M}) $.
The satisfaction relation $ \Gamma(U,\mathcal{M}) \models \varphi(\bvec{a}) $ is defined as usual.
The reader should notice that the relations $ \Gamma(U,\mathcal{M}) \models \varphi(\bvec{a}) $
and $ \dVert{\varphi(\bvec{a})} =U $ do not coincide in general.

Given a filter $ \mathfrak{f} $ on $ \mathcal{O}(X) $,
we write $ \Gamma(\mathcal{M}/\mathfrak{f}) $ for the set of global elements of 
the $ (\mathcal{O}(X)/\mathfrak{f}) $-valued $ \mathcal{L} $-structure $ \mathcal{M}/\mathfrak{f} $.
If $ \mathcal{M} $ satisfies the \textbf{Assumption}, then so does $ \mathcal{M}/\mathfrak{f} $,
and $ \Gamma(\mathcal{M}/\mathfrak{f}) $ becomes an $ \mathcal{L} $-structure.
The resulting structure $ \Gamma(\mathcal{M}/\mathfrak{f}) $ will play an essential role in describing our theorems.

Returning to the case of $ (\mathcal{M},\delta) = \Theta(P) $, we have the desired result.
\begin{Prop}
	$ \Gamma(\mathcal{M}/\mathfrak{f}) $ is isomorphic to the $ \mathcal{L} $-structure $ P/\mathfrak{f} $
	under the bijection in \textbf{Lemma \ref{lem:P/f-vs-Gamma(M/f)-as-sets}}.
\end{Prop}
\begin{proof}
	By the above constructions and \textbf{Definition \ref{def:sections-and-quotients-of-sheaves}}(2),
	\begin{align*}
		\Gamma(\mathcal{M}/\mathfrak{f}) \models R([\bvec{a}]_{\mathfrak{f}}) 
		& \defarrow R^{\mathcal{M}/\mathfrak{f}}([\bvec{a}]_{\mathfrak{f}}) := [R^{\mathcal{M}}(\bvec{a})]_{\mathfrak{f}} = [1_X]_{\mathfrak{f}}, \\
		& \iff R^{\mathcal{M}}(\bvec{a} ) = \bigvee \Set{W \leq \delta(\bvec{a}) \smid \bvec{a}|_W \in R^P(W)} \in \mathfrak{f}, \\
		& \iff \exists W \in \mathfrak{f},\, P(W) \models R(\bvec{a}|_W), \\
		& \iff P/\mathfrak{f} \models R([\bvec{a}]_{\mathfrak{f}}). \qedhere
	\end{align*}
\end{proof}
Thus, the construction of $ \Gamma(\mathcal{M}/\mathfrak{f}) $ indeed generalizes that of $ P/\mathfrak{f} $.
In the remainder of this section, let $ \mathcal{M} $ be an $ \mathcal{O}(X) $-valued $ \mathcal{L} $-structure satisfying the \textbf{Assumption}.

\subsection{{\L}o\'{s}'s Theorem}
{\L}o\'{s}-type theorems for sheaves of structures appeared in \cite[p.~179, Ultrastalk Theorem]{Ell1974} (see \S\ref{subsubsec:Ellerman's-viewpoint}),
\cite[Theorem 2.6 attributed to F.~Miraglia]{Brunner2016}, and \cite[Teorema 5.2]{Cai1995}.
The first two of them restrict themselves to $ \forall $-free formulas. Caicedo's result is closer to ours, but no proof is given there.
We give a generalization of {\L}o\'{s}'s theorem improving all these results,
and also give a characterization of Heyting-valued structures for which {\L}o\'{s}'s theorem holds w.r.t.\ any maximal filter,
which generalizes a similar theorem in \cite[Theorem 2.8]{PieVia2020} for Boolean-valued structures consisting of global elements only.

\begin{Def}
	For each $ \mathcal{L} $-formula $ \varphi $, the \emph{G\"{o}del translation} $ \varphi^G $ is defined inductively:
	\begin{itemize}
		\item $ \bot^G \equiv \bot $, and $ \varphi^G \equiv \lnot \lnot \varphi $ if $ \varphi$ is atomic but not $ \bot $.
		\item $ (\varphi \land \psi)^G \equiv \varphi^G \land \psi^G, \qquad (\varphi \lor \psi)^G \equiv \lnot ( \lnot \varphi^G \land \lnot \psi^G ), $
		\item $ (\varphi \limply \psi)^G \equiv \varphi^G \limply \psi^G, $
		\item $ (\forall v\varphi(v,\bvec{u}))^G \equiv \forall v \varphi^G(v,\bvec{u}), \qquad 
		(\exists v\varphi(v,\bvec{u}))^G \equiv \lnot \forall v \lnot \varphi^G(v,\bvec{u}) $. \qedhere
	\end{itemize}
\end{Def}

\begin{Def}
	A filter $ \mathfrak{f} $ on $ \mathcal{O}(X) $ is \emph{$ \mathcal{M} $-generic} when it satisfies the following:
	\begin{itemize}
		\item for each closed $ \mathcal{L}_{\mathcal{M}} $-formula $ \varphi(\bvec{a}) $ with $ \delta(\bvec{a}) \in \mathfrak{f} $,
		either $ \dVert{\varphi^G(\bvec{a})}^{\mathcal{M}} \in \mathfrak{f}$ or $ \dVert{\lnot \varphi^G(\bvec{a})}^{\mathcal{M}} \in \mathfrak{f} $ holds.
		
		\item for any $ \mathcal{L}_{\mathcal{M}} $-formula $ \varphi(v,\bvec{a}) $ with $ \delta(\bvec{a})\in\mathfrak{f} $,
		if $ \dVert{\exists v \varphi^G(v,\bvec{a})}^{\mathcal{M}} \in \mathfrak{f} $, 
		then there exists $ b \in \mathcal{M} $ such that $ \dVert{\varphi^G(b,\bvec{a})}^{\mathcal{M}} \in \mathfrak{f} $.
		\qedhere
	\end{itemize}
\end{Def}

\begin{Thm}[cf.~{\cite[Teorema 5.2]{Cai1995}}] \label{thm:main-Los-thm}
	If $ \mathfrak{f} $ is $ \mathcal{M} $-generic,
	then, for any $ \mathcal{L} $-formula $ \varphi(\bvec{v}) $ and $ \bvec{a} \in \mathcal{M}^n $ with $ \delta(\bvec{a}) \in \mathfrak{f} $,
	\[ \Gamma(\mathcal{M}/\mathfrak{f}) \models \varphi([\bvec{a}]_{\mathfrak{f}}) \iff \dVert{\varphi^G(\bvec{a})}^{\mathcal{M}} \in \mathfrak{f}. \]
\end{Thm}
\begin{proof}
	Let $ \Phi $ be the set of closed $ \mathcal{L}_{\mathcal{M}} $-formulas $ \varphi(\bvec{a}) $ with $ \delta(\bvec{a}) \in \mathfrak{f}$ 
	for which the above equivalence hold.
	We can easily see that $ \Phi $ contains atomic formulas and is closed under the logical connectives $ \land, \lor, \limply $.
	For example, to see that $ \varphi(\bvec{a}),\psi(\bvec{a}) \in \Phi $ implies $ (\varphi(\bvec{a}) \limply \psi(\bvec{a})) \in \Phi $,
	we only have to show
	\[ \delta(\bvec{a}) \wedge \left ( \dVert{\varphi^G(\bvec{a})} \imply \dVert{\psi^G(\bvec{a})} \right ) \in \mathfrak{f} \iff \text{either}\; 
	\dVert{\varphi^G(\bvec{a})} \notin \mathfrak{f} \;\text{or}\;  \dVert{\psi^G(\bvec{a})} \in \mathfrak{f}. \]
	This follows immediately from $ \mathcal{M} $-genericity.
	
	Suppose $ \varphi(b,\bvec{a}) \in \Phi $ for any $ b $ with $ \delta(b) \in \mathfrak{f} $.
	Since $ \lnot \lnot \exists v \varphi^G $ and $ \lnot \forall v \lnot \varphi^G $ are intuitionistically equivalent,
	\begin{align*}
		\dVert{\lnot \forall v \lnot \varphi^G(v,\bvec{a})} \in \mathfrak{f} & \iff \dVert{\lnot \lnot \exists v \varphi^G(v,\bvec{a})} \in \mathfrak{f},  \\
		& \iff \dVert{\exists v \varphi^G(v,\bvec{a})} \in \mathfrak{f}, \\
		& \iff \exists b \in \mathcal{M},\, \dVert{\varphi^G(b,\bvec{a})} \in \mathfrak{f}, \\
		& \iff \exists b \in \mathcal{M},\, \delta(b) \in \mathfrak{f} \;\text{and}\;
		\Gamma(\mathcal{M}/\mathfrak{f}) \models \varphi([b]_{\mathfrak{f}},[\bvec{a}]_{\mathfrak{f}}),\\
		& \iff \Gamma(\mathcal{M}/\mathfrak{f}) \models \exists v \varphi(v,[\bvec{a}]_{\mathfrak{f}}).
	\end{align*}
	
	For the universal quantifier, we need a fact on G\"{o}del translation.
	Since $ (\varphi\leftrightarrow\lnot\lnot\varphi)^G \equiv \varphi^G \leftrightarrow \lnot \lnot \varphi^G $ holds and 
	$ \varphi\leftrightarrow\lnot\lnot\varphi $ is classically valid, 
	$ \varphi^G \leftrightarrow \lnot \lnot \varphi^G $ is intuitionistically valid by \cite[Theorem 6.2.8]{vDalLS}.
	Therefore,
	\begin{align*}
		\dVert{\forall v \varphi^G(v,\bvec{a})} \in \mathfrak{f} & \iff \dVert{ \lnot \forall v \lnot \lnot \varphi^G(v,\bvec{a})} \notin \mathfrak{f},  \\
		& \iff \dVert{\exists v \lnot \varphi^G(v,\bvec{a})} \notin \mathfrak{f},  \\
		& \iff \forall b \in \mathcal{M},\, \delta(b) \in \mathfrak{f} \;\text{implies}\; \dVert{\lnot\varphi^G(b,\bvec{a})} \notin \mathfrak{f}, \\
		& \iff \forall b \in \mathcal{M},\, \delta(b) \in \mathfrak{f} \;\text{implies}\; \dVert{\varphi^G(b,\bvec{a})} \in \mathfrak{f}, \\
		& \iff \forall b \in \mathcal{M},\, \delta(b) \in \mathfrak{f} \;\text{implies}\; 
		\Gamma(\mathcal{M}/\mathfrak{f}) \models \varphi([b]_{\mathfrak{f}},[\bvec{a}]_{\mathfrak{f}}),\\
		& \iff \Gamma(\mathcal{M}/\mathfrak{f}) \models \forall v \varphi(v,[\bvec{a}]_{\mathfrak{f}}). \qedhere
	\end{align*}
\end{proof}

We say a formula is \emph{$ \forall $-free} if it is built up without $ \forall $.
\begin{Cor*}
	In the above notations, suppose that either of the following conditions holds:
	\begin{itemize}
		\item $ \mathcal{O}(X) $ is a complete Boolean algebra.
		
		\item $ \varphi $ is $ \forall $-free. (In particular, $ \varphi^G $ and $ \lnot \lnot \varphi $ are intuitionistically equivalent.)
	\end{itemize}
	Then, for any $ \mathcal{M} $-generic filter $ \mathfrak{f} $ and $ \bvec{a} \in \mathcal{M}^n $ with $ \delta(\bvec{a}) \in \mathfrak{f} $,
	\[ \Gamma(\mathcal{M}/\mathfrak{f}) \models \varphi([\bvec{a}]_{\mathfrak{f}}) \iff \dVert{\varphi(\bvec{a})}^{\mathcal{M}} \in \mathfrak{f}. \qedhere \]
\end{Cor*}

A key to finding $ \mathcal{M} $-generic filters is the following proposition.
For proofs, the reader is guided to refer \cite[Theorem 2.1]{Mira1988} and \cite[Teorema 3.3]{Cai1995}.
\begin{Prop*}[Maximum Principle] \label{prop:max-principle-for-sheaves}
	If $ \mathcal{M} $ is complete as an $ \mathcal{O}(X) $-valued set, then, 
	for any $ \mathcal{L}_{\mathcal{M}} $-formula $ \varphi(v,\bvec{a}) $, there exists $ b \in \mathcal{M} $ such that
	\[ \dVert{\varphi(b,\bvec{a})}^{\mathcal{M}} \leq \dVert{\exists v \varphi(v,\bvec{a})}^{\mathcal{M}} \leq 
	\dVert{\lnot\lnot\varphi(b,\bvec{a})}^{\mathcal{M}} \quad \text{in}\quad \mathcal{O}(X).\]
	We say $ \mathcal{M} $ satisfies the \emph{maximum principle} if the conclusion holds.
\end{Prop*}

In the topological case, the maximum principle means that we can find an open set $ \dVert{\varphi(b,\bvec{a})} $ dense in $ \dVert{\exists v \varphi(v,\bvec{a})} $.

\begin{Rmk}
	Volger \cite[p.~4]{Vol1976} pointed out that the maximum principle for \textit{Boolean}-valued structures holds under a weaker assumption:
	
	{\centering for any $ \{a_i\}_{i \in I} \subseteq \mathcal{M} $ and any (strong) anti-chain $ \{U_i\}_{i \in I} \subseteq \mathcal{O}(X) $ 
		
		(i.e., a pairwise disjoint family) satisfying $ U_i \leq \delta(a_i)$ for each $ i \in I $,
		
		there exists $ a \in \mathcal{M} $ such that $ U_i \leq \delta(a,a_i) $ for each $ i \in I $.
		\par}
	
	\noindent For detailed proof, see \cite[Proposition 2.11]{PieVia2020}, where the authors call this the \emph{mixing property}.
	This does not assume any existence of restrictions of elements,
	and we would like to remove such an assumption from the previous proposition.
	However, we cannot apply their argument to Heyting-valued structures 
	because the anti-chain they consider may not cover $ \dVert{\exists v \varphi} $ in general.
	Bell \cite{BellIST} assumes that the frame in consideration is \textit{refinable}
	to ensure existence of an anti-chain refining $ \dVert{\exists v \varphi} $ and
	to show that a specific Heyting-valued structure satisfies the maximum principle (he calls it the Existence Principle).
	We do not know whether the existence of restrictions and refinements can be removed from the previous proposition.
	
	We also remark that all the results mentioned above on the maximum principle involve the use of the axiom of choice or its equivalents.
\end{Rmk}

\begin{Thm}[Main Theorem] \label{thm:char-of-MP-and-Los}
	\ 
	
	For any $ \mathcal{O}(X) $-valued $ \mathcal{L} $-structure $ \mathcal{M} $ satisfying the \textbf{Assumption}, TFAE:
	\begin{enumerate}[label=(\roman*)]
		\item $ \mathcal{M} $ satisfies the following variant of the maximum principle:
		for any $ \mathcal{L}_{\mathcal{M}} $-formula $ \varphi(v,\bvec{a}) $, there are finitely many $ b_1,\dots,b_r \in \mathcal{M} $ such that
		\[ \bigvee_i \dVert{ \varphi^G(b_i,\bvec{a}) }^{\mathcal{M}} \leq \dVert{\exists v \varphi^G(v,\bvec{a})}^{\mathcal{M}} 
		\leq \lnot\lnot \bigvee_i \dVert{ \varphi^G(b_i,\bvec{a}) }^{\mathcal{M}}.\]
		\item Every maximal filter on $ \mathcal{O}(X) $ is $ \mathcal{M} $-generic.
		\item For any maximal filter $ \mathfrak{m} $ on $ \mathcal{O}(X) $ and
		any closed $ \mathcal{L}_{\mathcal{M}} $-formula $ \varphi(\bvec{a}) $ with $ \delta(\bvec{a}) \in \mathfrak{m} $,
		\[ \Gamma(\mathcal{M}/\mathfrak{m}) \models \varphi([\bvec{a}]_{\mathfrak{m}}) \iff \dVert{\varphi^G(\bvec{a})}^{\mathcal{M}} \in \mathfrak{m}. \]
	\end{enumerate}
\end{Thm}
\begin{proof}
	\noindent (i)$ \Rightarrow $(ii):
	let $ \mathfrak{m} $ be a maximal filter on $ \mathcal{O}(X) $.
	For any $ U \in \mathcal{O}(X)$, either $ U \in \mathfrak{m} $ or $ \neg U \in \mathfrak{m} $ holds.
	Moreover, if $ U \vee V \in \mathfrak{m} $, then $ U \in \mathfrak{m} $ or $ V \in \mathfrak{m} $.
	Thus, the maximum principle implies $ \mathcal{M} $-genericity of $ \mathfrak{m} $.
	
	\noindent (ii)$ \Rightarrow $(iii):
	by \textbf{Theorem \ref{thm:main-Los-thm}}.
	
	\noindent (iii)$ \Rightarrow $(i):
	the following argument is a modification of the proof of \cite[Theorem 2.8]{PieVia2020}.
	To simplify notations, we may assume $ \delta(\bvec{a}) = 1_X $ and suppress the parameter $ \bvec{a} $.
	For an arbitrary $ \bvec{a} $, 
	we may use the frame $ \mathcal{O}(\delta(\bvec{a})) = (\delta(\bvec{a})){\downarrow} $ instead of $ \mathcal{O}(X) $ in the following.
	
	For any $ \mathcal{L}_{\mathcal{M}} $-formula $ \varphi(v) $ with $ \dVert{\exists v \varphi^G(v)} \neq 0_X $,
	we can take a maximal filter $ \mathfrak{m} \ni \dVert{\exists v \varphi^G(v)} $.
	Since $ \exists v \varphi^G \limply \lnot \forall v \lnot \varphi^G $ is intuitionistically valid,
	we have $ \dVert{(\exists v \varphi(v))^G} \in \mathfrak{m} $.
	By the assumption, $ \Gamma(\mathcal{M}/\mathfrak{m}) \models \exists v \varphi(v) $.
	Then there exists $ b \in \mathcal{M} $ such that $ \delta(b) \in \mathfrak{m} $ and
	$ \Gamma(\mathcal{M}/\mathfrak{m}) \models \varphi([b]_{\mathfrak{m}}) $.
	Again by the assumption, there exists $ b \in \mathcal{M} $ such that $ \dVert{\varphi^G(b)} \in \mathfrak{m} $.
	
	We have just shown that any maximal filter containing $ \dVert{(\exists v \varphi(v))^G} $ also contains some $ \dVert{\varphi^G(b)} $.
	Notice that $ \dVert{\varphi^G(b)} $ is a regular element of $ \mathcal{O}(\delta(b)) $
	because $ \varphi^G \leftrightarrow \lnot \lnot \varphi^G $ is intuitionistically valid.
	We write $ \mathrm{Reg}(\mathcal{O}(X)) $ for the complete Boolean algebra of regular elements of $ \mathcal{O}(X) $.
	Now we consider the spectrum $ \mathrm{Spec}(\mathrm{Reg}(\mathcal{O}(X))) $ of $ \mathrm{Reg}(\mathcal{O}(X)) $,
	i.e., the Stone space of ultrafilters on $ \mathrm{Reg}(\mathcal{O}(X)) $ whose basic (closed) open sets are of the form
	\[ D(U) := \Set{\mathfrak{u} \in \mathrm{Spec}(\mathrm{Reg}(\mathcal{O}(X))) \smid \mathfrak{u} \ni U } 
	\quad \text{for $ U \in \mathrm{Reg}(\mathcal{O}(X))$.} \]
	Since maximal filters on $ \mathcal{O}(X) $ correspond to ultrafilters on $ \mathrm{Reg}(\mathcal{O}(X)) $
	(see \cite[Exercise II.4.9]{JohSS}, \cite[Theorem 1.44]{Sipos}), the above observation yields
	\footnote{While $ \dVert{\varphi^G(b)} $ is regular in $ \mathcal{O}(\delta(b)) $, it is not necessarily regular in $ \mathcal{O}(X) $.
		This is why we use $ \neg \neg \dVert{\varphi^G(b)} $ here.}
	\[ D\left (\dVert{(\exists v \varphi(v))^G}\right ) \subseteq \bigcup_{b \in \mathcal{M}} D\left (\neg\neg\dVert{\varphi^G(b)}\right ). \]
	By compactness of $ D(\dVert{(\exists v \varphi(v))^G}) $, we can find $ b_1,\dots,b_r $ such that
	\[ D\left (\dVert{(\exists v \varphi(v))^G}\right ) \subseteq \bigcup_{i} D\left (\neg\neg\dVert{\varphi^G(b_i)}\right )
	=D \left (\textstyle \neg\neg \bigvee_i \dVert{ \varphi^G(b_i) } \right ). \]
	Hence, we have $ \dVert{\exists v \varphi^G(v)} \leq \dVert{(\exists v \varphi(v))^G} 
	\leq \neg\neg \bigvee_i \dVert{ \varphi^G(b_i)} $.
\end{proof}

Combining the results in this section, we obtain
\begin{Cor}[Classical {\L}o\'{s}'s theorem]
	Let $ X $ be a set, $ \{\mathcal{M}_x\}_{x \in X} $ an $ X $-indexed family of non-empty $ \mathcal{L} $-structures, and $ \mathfrak{u} $ an ultrafilter over $ X $.
	Then, for any $ \mathcal{L} $-formula $ \varphi(u_1,\dots,u_n) $ and $ a^1,\dots,a^n \in \prod_x \mathcal{M}_x $,
	\begin{align*}
		& \phantom{\iff}\;\prod_x \mathcal{M}_x /\mathfrak{u} \models \varphi([a^1]_{\mathfrak{u}},\dots,[a^n]_{\mathfrak{u}}) \\
		& \iff \Set{ x \in X \smid \mathcal{M}_x \models \varphi(a^1_x,\dots,a^n_x) } \in \mathfrak{u}. \qedhere
	\end{align*}
\end{Cor}
\begin{proof}
	Let $ P $ be the sheaf corresponding to the local homeomorphism $ \coprod_{x \in X} \mathcal{M}_x \to X $ as in \textbf{Example \ref{expl:products}}.
	The statement follows from the facts $ \prod_x \mathcal{M}_x /\mathfrak{u} \simeq P/\mathfrak{u} \simeq \Gamma(\Theta(P)/\mathfrak{u})$ and
	$ \dVert{\varphi(a^1,\dots,a^n)}^{\Theta(P)} =  \Set{ x \in X \smid \mathcal{M}_x \models \varphi(a^1_x,\dots,a^n_x) } $.
\end{proof}

We remark that Pierobon \& Viale \cite{PieVia2020} give set-theoretic examples of
\begin{itemize}
	\item a Boolean-valued structure which is not a sheaf but satisfies the maximum principle, and 
	\item a Boolean-valued structure violating {\L}o\'{s}'s theorem (and the maximum principle).
\end{itemize}

\subsubsection{Ellerman's Viewpoint} \label{subsubsec:Ellerman's-viewpoint}
Various {\L}o\'{s}-type theorems for specific sheaves of structures have been considered in the literature.
Some of them are special cases of our theorem, but others are not.
For simplicity, we treat $ \forall $-free formulas only.
Let $ X $ be a topological space and $ \mathrm{Spec}(X) $ be the space of prime filters on the frame $ \mathcal{O}(X) $.
$ \mathrm{Spec}(X) $ has the basic open set $ D_U=\Set{\mathfrak{p} \smid U \in \mathfrak{p}} $ for each $ U \in \mathcal{O}(X) $.
We have a continuous map $ \eta \colon X \to  \mathrm{Spec}(X) $ sending $ x $ to $ \mathfrak{n}_x $.
For any sheaf $ P $ of $ \mathcal{L} $-structures, the direct image sheaf $ \eta_*P $ on $ \mathrm{Spec}(X) $ is again a sheaf of $ \mathcal{L} $-structures.
Ellerman \cite[p.~179]{Ell1974} showed the following (cf.~\cite{Mulv1977} and \cite{Sipos}):
\begin{Thm*}[Ultrastalk Theorem]
	For any maximal filter $ \mathfrak{m} \ni U $ and 
	any closed $ \mathcal{L}_{\Theta(P)} $-formula $ \varphi(\bvec{a}) $ with $ \bvec{a} \in P(U)^n=(\eta_*P)(D_U)^n $,
	\[ (\eta_*P)_{\mathfrak{m}} \models \varphi([\bvec{a}]_{\mathfrak{m}}) \iff \dVert{\varphi(\bvec{a})} \in \mathfrak{m}. \qedhere\]
\end{Thm*}
Our theorem subsumes the Ultrastalk Theorem since
\[ (\eta_*P)_{\mathfrak{m}} = \rlim_{D \ni \mathfrak{m}} (\eta_*P)(D) \simeq \rlim_{D_U \ni \mathfrak{m}} (\eta_*P)(D_U)
=\rlim_{U \in \mathfrak{m}} P(U) = P/\mathfrak{m}. \]
Especially, {\L}o\'{s} theorem for unbounded Boolean ultrapowers \cite[Theorem 1.5]{Mans1971} is under our scope (cf.~\cite{Macnab1977}).
However, Ellerman's approach suggests a significant viewpoint missing in ours:
various model-theoretic constructions are realized by taking stalks of \textit{sheaves on the spectrum of a distributive lattice}.
For example, as we saw in \textbf{Example \ref{expl:bdd-boolean-ultrapower}},
a bounded Boolean ultrapower is a stalk over an ultrafilter on a (possibly non-complete) Boolean algebra.
There are {\L}o\'{s}-type theorems for such structures, e.g.,
\cite[Lemma 7.1]{BurWer1979} for \textit{a family of} Boolean products.
The relationship between these theorems and our approach should be explored elsewhere (see the comments in the next section).

\section{Related Topics and Future Directions} \label{sec:conclusion}

Finally, we give an overview of various sheaf-theoretic methods in model theory with an expanded list of previous works,
and indicate future directions from a topos-theoretic perspective.

\paragraph{Forcing and Generic Models:}
We again assume all formulas are $ \forall $-free.
Let $ P $ be a sheaf of $ \mathcal{L} $-structures on a topological space $ X $.
As we noticed in \S\ref{subsec:forcing-val}, forcing values give the sheaf semantics in $ \mathbf{Sh}(X) $. 
We can consider another forcing relation, for $ x \in X $,
\[ x \Vdash_P \varphi(\bvec{a}) \defarrow x \in \dVert{\varphi(\bvec{a})}.  \]
Caicedo \cite{Cai1995} called ``$ U \Vdash $'' the local semantics and ``$ x \Vdash $'' the punctual semantics.

On the other hand, each stalk $ P_x $ is an $ \mathcal{L} $-structure, and we can also consider the relation $ P_x \models \varphi(\bvec{a}_x) $
for each closed $ \mathcal{L}_{\mathcal{M}} $-formula $ \varphi(\bvec{a}) $ with $ x \in \delta(\bvec{a}) $.
Define the \emph{discrete value} of a formula:
\[ \dvert{\varphi(\bvec{a})} := \Set{x \in \delta(\bvec{a}) \smid P_x \models \varphi(\bvec{a}_x)}. \]
For any atomic relation $ R $, by definition,
\[ P_x \models R(\bvec{a}_x) \iff \exists V \ni x,\, P(V) \models R(\bvec{a}|_V), \]
i.e., $ \dvert{R(\bvec{a})} = \dVert{R(\bvec{a})}$.
However, in general, $ \dvert{\varphi(\bvec{a})} \neq \dVert{\varphi(\bvec{a})} $.
Some authors considered the relationship between them (\cite[\S1]{Mans1977} and \cite[Theorem 4.3, Lemma 5.1]{Lou1979}).

Kaiser \cite{Kai1977} addressed the problem when the relations $ P_x \models \varphi(\bvec{a}_x) $ and  $ x \Vdash_P \varphi(\bvec{a}) $ coincide for any formula.
He called such $ P_x $ a \textit{generic stalk}.
If the filter $ \mathfrak{n}_x $ is $ \Theta(P) $-generic, then $ P_x $ is a generic stalk by our {\L}o\'{s}-type theorem.
Kaiser used generic stalks to obtain omitting types and consistency results similar to those in \cite{Kei1973} (cf.~\cite[\S6]{Cai1995}, \cite{BruMira2004}).

From a topos-theoretic perspective, Blass \& Scedrov \cite{BlaSce1983} constructed the classifying toposes of existentially closed models and finite-generic models.
Their work was apparently inspired by Keisler's viewpoint \cite{Kei1973} and might be related to ours.

\paragraph{Stalks, Global Sections, and Induced Geometric Morphisms:}
In addition to stalks of sheaves, the structure $ \Gamma(X,P) $ of global sections is of our future interest (see below).
The Feferman--Vaught theorem works for global sections just like {\L}o\'{s}'s theorem does for stalks.
Comer \cite{Comer1974} gave a sheaf-theoretic interpretation of the original Feferman--Vaught theorem \cite{FefVau1959}.
Feferman--Vaught type theorems and their applications to sentences preserved under taking global sections were pursued in
\cite{Vol1976}, \cite{LavLuc1985}, \cite{Mans1977}, \cite{Taka1980} and \cite{BurWer1979} (cf.~\cite{Vol1979}).

From a topos-theoretic viewpoint, taking stalks and global sections can be seen as part of \textit{geometric morphisms}.
Any morphism $ f \colon X \to Y $ of locales (\textbf{Definition \ref{def:cat-of-locales}}) or of topological spaces
induces a geometric morphism $ (f^*,f_*) \colon \mathbf{Sh}(X) \to \mathbf{Sh}(Y) $. Then,
\begin{itemize}
	\item The stalk $ P_x $ is $ f^*P $ for the geometric morphism $ \mathbf{Set} \to \mathbf{Sh}(X) $ induced by the point $ f=x \colon 1 \to X $.
	\item The set $ P(X) $ is $ f_*P $ for the (essentially unique) geometric morphism $ \mathbf{Sh}(X) \to \mathbf{Set} $ induced by $ f \colon X \to 1 $.
\end{itemize}
Furthermore, we can construct a geometric morphism $ \mathbf{Set}(\mathcal{O}(X)) \to \mathbf{Set}(\mathcal{O}(Y)) $,
and it is canonically identified with $ (f^*,f_*) $ via the equivalence in \textbf{Corollary \ref{cor:Sh(X)=Set(OX)}}.
Therefore, we may investigate stalks and global sections in the more general framework of base change of Heyting-valued structures.
This categorical approach has an advantage over the set-theoretic approach of \cite{ACM2019} to base change of Heyting-valued universes
since the construction of geometric morphisms is much simpler and
the logical behavior under base change along them is well-understood for various classes of morphisms of locales \cite[Chapter C3]{Elephant}.

\paragraph{Sheaf Representation and Model Theory for Sheaves:}
Algebraic structures often have representations as global sections of sheaves of structures.
Knoebel's monograph \cite{KnoebelSABS} includes a brief description of a history of sheaf representations of algebras (see also \cite[Chapter V]{JohSS}).
Sheaf representations over Stone spaces, e.g., Pierce representation of commutative rings \cite{Pierce1967}, play a special role in model theory.
Following Lipshitz \& Saracino \cite{LipSara1973}, Macintyre \cite{Mac1973} established a general method
for obtaining model-companions of theories whose models have sheaf representations over Stone spaces with good stalks (cf.~\cite{Cars1973}).
He exploited Comer's version of the Feferman--Vaught theorem to transfer model-theoretic properties of stalks to global sections.
This line of research was followed by \cite{Weis1975}, \cite{vdD1977}, \cite{Comer1976} and \cite{BurWer1979} (see also \cite[\S6]{Mac1977}).
Later, Bunge \& Reyes \cite{BunRey1981} gave a topos-theoretic unification (cf.~\cite{Bun1981}).

In this line of research, sheaves having good stalks are often sheaf models of well-behaved theories.
For example, any (commutative)  von~Neumann regular ring $ R $ is represented by a sheaf of rings over a Stone space $ X(R) $ whose stalks are fields,
and such a sheaf is a model of the theory of fields in the topos $ \mathbf{Sh}(X(R)) $.
The theory of von~Neumann regular rings has the model-completion, whose models are represented by ``algebraically closed fields'' in sheaf toposes over Stone spaces.
Thus, we may expect that developing model theory for sheaves will deepen our understanding of ordinary model theory.
Model theory for sheaves has been studied intermittently by some authors.
The pioneering work is \cite{Lou1979}, where Loullis had already pointed out the importance of the viewpoint we just mentioned.
Our standpoint emphasizing Heyting-valued structures was greatly influenced by him too.
Some other authors considered model-theoretic phenomena for models in various toposes (\cite{Bell1981}, \cite{Zawa1983}, \cite{GeoVoi1985}, \cite{Mira1988}, \cite{Ack2014}).

In fact, model theory for sheaves is part of what should be called \emph{topos-internal model theory} or \emph{model theory in toposes}.
Topos-internal model theory concerns theories internal to toposes,
and internal theories in a sheaf topos admit \textit{sheaves of function symbols and relation symbols} (cf.~\cite{SJHenryThesis}).
It must be closer to doing model theory in a Heyting-valued universe (cf.~\cite{KusKutaBVA}).
The approach by Brunner \& Miraglia \cite{Brunner2016}, admitting a presheaf of constant symbols in place of a set of constants,
is regarded as a restricted form of topos-internal model theory.
In contrast to the scarcity of research on topos-internal model theory,
there is much more on universal algebra in toposes and sheaf models for constructive mathematics.

Finally, we would like to mention a potential application of topos-internal model theory to algebraic geometry.
At the end of \cite{Lou1979}, Loullis suggests that algebraic geometry over von~Neumann regular rings \cite{SaraWeis1975} 
could be obtained by doing algebraic geometry in some topos.
The works of Bunge \cite{Bun1982} and her student MacCaull \cite{MacCaull1988} reflect that idea, but no one followed them.
We leave that direction as the ultimate goal of our research.

\end{document}